\documentclass[11pt, reqno]{amsart}
\usepackage[utf8]{inputenc}
\usepackage{amsmath, amsthm, amsfonts, amssymb, color}
\usepackage{setspace}
\usepackage{mathrsfs}
\usepackage{amsfonts, amsmath}
\usepackage{amsmath,amstext,amsthm,amssymb,amsxtra}
\usepackage{newunicodechar}
\newunicodechar{，}{,}

\usepackage{hyperref}
\hypersetup{colorlinks=true,linkcolor=black, urlcolor=blue, citecolor=blue}

\usepackage{color}

\allowdisplaybreaks

\newcommand{\EQ}[1]{\begin{equation}\begin{split} #1 \end{split}\end{equation}}
\newcommand{\EQN}[1]{\begin{equation*}\begin{split} #1 \end{split}\end{equation*}}
\newcommand{\CAS}[1]{\begin{cases} #1 \end{cases}}

\newcommand{\case}[1]{
\EQ{
\begin{array}{l}
\left\{
\begin{aligned}
#1
\end{aligned}
\right.
\end{array}
}
}

\usepackage[margin=3cm, a4paper]{geometry}

\definecolor{deepgreen}{cmyk}{1,0,1,0.5}

\newcommand {\norm}[1]{\left\|#1\right\|}

\newtheorem{theorem}{Theorem}[section]

\newtheorem{prop}[theorem]{Proposition}
\newtheorem{lemma}[theorem]{Lemma}

\theoremstyle{remark}
\newtheorem{remark}[theorem]{Remark}
\newtheorem{rema}[theorem]{Remark}

\theoremstyle{definition}

\allowdisplaybreaks

\def\th2{\frac{\theta}{2}}

\def\normb#1{\big\|#1\big\|}

\def\norm#1{\|#1\|}

\def\wh#1{\widehat{#1}}

\newcommand{\Div}{{\rm div\,}}

\newcommand{\R}{{\mathbb R}}

\newcommand{\Z}{{\mathbb Z}}
\newcommand{\X}{{\mathbb X}}
\newcommand{\ft}{{\mathcal{F}}}
\newcommand{\Hl}{{\mathcal{H}}}

\newcommand{\les}{{\lesssim}}

\newcommand{\supp}{{\mbox{supp}}}

\numberwithin{equation}{section}

\begin{document}

\title[3D Compressible Navier-Stokes equation]{
Global well-posedness for the 3D compressible Navier-Stokes equations in optimal Besov space
}

\author[Z. Guo]{Zihua Guo}
\address{School of Mathematics, Monash University, Melbourne, VIC 3800, Australia}
\email{zihua.guo@monash.edu}

\author[Z. Song]{Zihao Song}
\address{School of Mathematics, Nanjing University of Aeronautics and Astronautics, Nanjing 211106,
Jiangsu, China}
\email{szh1995@nuaa.edu.cn}

\author[M. Yang]{Minghua Yang}
\address{School of Information Management and Mathematics, Jiangxi University of Finance and
Economics, Nanchang, 330032, China}
\email{minghuayang@jxufe.edu.cn}

\thispagestyle{empty}

\begin{abstract}
We consider the Cauchy problem to the 3D barotropic compressible Navier-Stokes equation. We prove global well-posedness, assuming that the initial data $(\rho_0-1,u_0)$ has small norms in the critical Besov space $\X_p=\dot{B}_{p,1}^{3/p}(\R^3)\times \dot{B}_{p,1}^{-1+3/p}(\R^3)$ for $2\leq p<6$ and $(\rho_0-1,\rho_0u_0)$ satisfies an additional low frequency condition. Our results extend the previous results in \cite{FD2010, CMZ2010, H20112} where $p<4$ is needed for high frequency, to the optimal range $p<6$. The main ingredients of the proof consist of: a novel nonlinear transform that uses momentum formulation for low-frequency and effective velocity method for high frequency, and estimate of parabolic-dispersive semigroup that enables a $L^q$-framework for low frequency.
\end{abstract}

\subjclass[2020]{35Q30, 76N06}
\keywords{Compressible Navier-Stokes equations, Well-posedness, Critical spaces}

\maketitle

\tableofcontents

\section{Introduction}

In this paper, we consider the Cauchy problem to the barotroptic compressible Navier-Stokes equations:
\begin{equation}\label{cns1}
\left\{
\begin{aligned}
& \partial_{t}\rho+{\rm div}\, (\rho u)=0,\quad  t>0,\,  x\in \mathbb{R}^{d},\\
&  \partial_{t}(\rho u)+{\rm div}\, (\rho u\otimes u)-\mathcal{A}u+\nabla P(\rho)=0,\quad  t>0,\,  x\in \mathbb{R}^{d},\\
& \rho(0,x)=\rho_{0}(x),\,   u(0,x)=u_{0}(x),\quad x\in \mathbb{R}^{d}.  \\
\end{aligned}
\right.
\end{equation}
Here, the two unknown functions $\rho(t,x)$ and $u(t,x)$ represent the fluid's density and velocity, respectively. We will also use the momentum $m=\rho u$. $P(\rho)$ represents the pressure of the fluid  depending on the density and $\mathcal{A} = \mu\Delta + (\mu + \lambda)\nabla \mathrm{div}$ is the Lamé operator, representing the viscosity, where the constants $\mu$ and $\lambda$ satisfy the elliptic conditions $\mu > 0$ and $2\mu +\lambda > 0$.

The Cauchy problem \eqref{cns1} is of fundamental importance (see \cite{PL1998, RH2015} and the references therein).
The study of the compressible barotropic Navier-Stokes equations has seen significant growth in past decades. Nash \cite{Nash} proved the local existence and uniqueness of smooth solutions. Matsumura-Nishida \cite{MN1979,MN1980} proved the global well-posedness for solutions near the constant equilibrium with small initial data in $H^s$.  Lions \cite{PL1998} established the global existence of weak solutions with large initial data. Xin \cite{X1998} proved the existence of blow-up solutions. We do not attempt to exhaust the list of related studies. 

We assume that the fluid density $\rho$ is a small perturbation near $1$. Let $a=\rho-1$. To simplify our presentation, we assume that 
\EQ{
\gamma = P'(1)=&\nu=2\mu +\lambda=1\\
I(a) = \frac{a}{a+1}, G'(a) =&\frac{ P'(a+1)}{a+1},k(a)=G'(a)-G'(0)
}
which allow us to rewrite \eqref{cns1} as:
\begin{equation}\label{CNS}
\left\{
\begin{aligned}
& \partial_ta+\Div u=-\Div(au),\\
& \partial_{t}u-\mathcal{A}u+\nabla a=-u\cdot \nabla u-I(a)\mathcal{A}u-k(a)\nabla a
\\
& a(0,x)=a_{0}(x),\,   u(0,x)=u_{0}(x).  \\
\end{aligned}
\right.
\end{equation}
Scaling invariance is important to achieve global-in-time results. This methodology originates in the pioneering work of Fujita and Kato \cite{FK1964} on the classical incompressible Navier-Stokes equations. For barotropic fluids, the scaling transformation
\begin{equation*}
\begin{aligned}
(a_{0}, u_{0})\rightsquigarrow(a_{0}(\lambda x),\lambda u_{0}(\lambda x)),\quad(a(t,x), u(t,x))
\rightsquigarrow(a(\lambda^{2}t, \lambda x),\lambda u(\lambda^{2}t, \lambda x)), \quad \lambda>0.
\end{aligned}
\end{equation*}
leaves equation \eqref{CNS} invariant when the pressure law $P$ is replaced by $\lambda^{2}P$. 
A natural choice of $\X$ which is scaling invariant for \eqref{CNS} is the critical Besov space
\begin{equation}\label{criticalspace}
(a_{0},\, u_{0})\in \X_p:=\dot{B}_{p,1}^{\frac{d}{p}}\times\dot{B}_{p,1}^{-1+\frac{d}{p}}.
\end{equation}

There are much literature studying \eqref{CNS} with initial data in $\X_p$. For local results,
Danchin  \cite{D2001,D2005} and Chen-Miao-Zhang \cite{CMZ2010R} established local existence for $1\leq p<2d$ and uniqueness for $1\leq p\leq d$. The uniqueness for $d<p<2d$ was proved in \cite{D2014} via a Lagrangian approach for \eqref{CNS}. 
On the other hand, the range $p<2d$ is optimal for well-posedness in $\X_p$. Ill-posedness in the sense that the continuity of the solution map $S_T$ fails at the origin in $\X_p$ was proved by Chen-Miao-Zhang \cite{CMZ2015} for $p>2d$, and then by Iwabuchi-Ogawa \cite{IO2022} for $p=2d$.  Recently, the first and third authors-Zhang \cite{GYZ2024} obtained the continuity of the solution map $S_T$ from $\X_p \to C_T\X_p$ for $1\leq p<2d$ via a combination of the Langrangian approach and the method of frequency envelope.

For global results, in a seminal paper \cite{D2000}, by delicate energy method, Danchin proved global existence and uniqueness for small initial data 
$(a_0,u_0)\in ( \dot{B}_{2,1}^{d/2-1} \cap \dot{B}_{2,1}^{d/2}) \times \dot{B}_{2,1}^{d/2-1}$.
It turns out that it is natural to use hybrid Besov space for the density, as the linearized equation has different behaviours on high and low frequency. Additional conditions on the low frequency seems necessary. Later, the results were extended to $L^p$-framework by   Charve--Danchin \cite{FD2010}, Chen--Miao--Zhang \cite{CMZ2010}, and Haspot \cite{H20112}, proving global existence and uniqueness with 
initial data satisfying
\EQ{
\|(a_{0},u_{0})^\ell\|_{\dot{B}_{2,1}^{d/2-1}} 
+ \|(a_{0},u_0)\|_{\X_p} \ll 1,
}
where \(p\) satisfies $2 \leq p \leq \min\left\{4, 
\frac{2d}{d-2}\right\}$ for $d \geq 3$ or $2 \leq p < 4$ for $d = 2$.
Note that uniqueness was restricted to $2\leq p\leq d$ in \cite{FD2010,CMZ2010}, and was improved in \cite{H20112} by introducing an effective velocity. 
Here for function $f$, $f^\ell$ (and $f^h$) means the low-frequency (high-frequency) component (see \eqref{hilo}). In the $L^p$-framework, the energy method does not work, and some novel methods were developed. On the other hand, it contains some physically interesting data. In particular for the case $p>d$, the regularity index is negative, and it could include some highly oscillating initial data (see Remark \ref{rk:highosci}). However, compared to the local results, the range of $p$ for global results is much smaller. 
To the best of our knowledge, the global results in the optimal Besov space in three and higher dimensions remain open problems. 

The purpose of this paper is to address this problem in three dimensions. The main result of this paper is

\begin{theorem}\label{gwp}

Assume $d=3$, $2 \leq p< 6,\,\,2 \leq q \leq p \leq 2q$, and 
\begin{align}\label{condition}
\frac{3}{q}-\frac{3}{p}\leq 1<\frac{2}{q}+\frac{3}{p}.
\end{align}
Assume $0<\delta\ll 1$. Assume $(a_{0},u_{0})\in E_\delta$ where (note that $m=(a+1)u$) 
\EQ{
E_\delta=\{(a_0,u_0) \in \X_{p}: \|(a_{0},m_{0})^{\ell}\|_{\dot{B}_{q,1}^{-3
+\frac{7}{q}}}+\|(a_{0},u_0)\|_{\X_p}\leq \delta\}.
}
Then the Cauchy problem \eqref{CNS} has a unique global solution $(a,u)\in X_{q,p}$ satisfying $\norm{(a,u)}_{X_{q,p}}\leq C\delta$ where $X_{q,p}$ is the Banach space defined by the norm
\begin{equation}\label{bound}
\begin{aligned}
\|(a,u)\|_{X_{q,p}}:=&\|(a,u)^{\ell}\|_{\tilde{L}_{t}^{\infty}(\dot{B}_{q,1}^{-1
+\frac{3}{q}})\cap\tilde{L}_{t}^{2}(\dot{B}_{q,1}^{-1+\frac{5}{q}})\cap\tilde{L}_{t}^{1}(\dot{B}_{q,1}^{\frac{5}{q}})}\\
&+\|a^{h}\|_{\tilde{L}_{t}^{\infty}(\dot{B}_{p,1}^{\frac{3}{p}})\cap\tilde{L}_{t}^{1}(\dot{B}_{p,1}^{\frac{3}{p}})}+
\|u^{h}\|_{\tilde{L}_{t}^{\infty}(\dot{B}_{p,1}^{-1+\frac{3}{p}})
\cap\tilde{L}_{t}^{1}(\dot{B}_{p,1}^{1+\frac{3}{p}})}.
\end{aligned}
\end{equation}
Moreover, $(a,u)^{\ell}\in {C}_{b}(\mathbb{R}_{+};\dot{B}_{q,1}^{-1+\frac{3}{q}})$, $(a,u)^{h}\in {C}_{b}(\mathbb{R}_{+};\X_p)$, $(a,m)^\ell \in C(\R_+: \dot{B}_{q,1}^{-3
+\frac{7}{q}})$, and
\EQ{
&\|(a,m)^\ell\|_{\dot{B}_{q,1}^{-3
+\frac{7}{q}}}\leq C(t), \quad \forall t>0,
}
and, for any $T>0$, the solution map $S_T: E_\delta\subset \X_p\to X_{q,p}\subset C_T\X_p$ is continuous.
\end{theorem}

\begin{remark}
We will show in Proposition~\ref{momentum} that, $m^\ell$ and $u^\ell$ have equivalent norms in the resolution space $X_{q,p}$. In particular when \( q = 2 \), \(\|(a_0, m_0)^{\ell}\|_{\dot{B}_{q,1}^{-3 + 7/q}}\sim \|(a_0, u_0)^{\ell}\|_{\dot{B}_{q,1}^{-3 + 7/q}}\). Thus, by taking $q=2$, Theorem \ref{gwp} covers the previous results obtained in \cite{FD2010, CMZ2010, H20112}. By taking $q\geq 3$, Theorem \ref{gwp} obtains the full range \(  p < 6 \) for high frequency.
\end{remark}

\begin{remark}
Theorem~\ref{gwp} applies to initial data $(a_{0},u_{0})\in \X_{p}$ considered  in \cite{D2014}. 
We provide the following examples:
\begin{enumerate}
    \item $\rho_0 = 1$, $m_0 = u_0$ with $u_0^\ell \in \dot{B}_{q,1}^{-3 + \frac{7}{q}}$;
    \item $a_0 = \mathcal{F}^{-1}(\chi(\xi))$, $u_0 = 
\mathcal{F}^{-1}\left(\psi\big(\tfrac{\xi}{N}\big)\right)$ for $N \gg 1$.
\end{enumerate}
One readily verifies that $m^\ell_0 \in \dot{B}_{q,1}^{-3 + \frac{7}{q}}$ 
whereas  $(a_{0},u_{0})\in \X_{p}$
\end{remark}
\begin{remark}\label{rk:highosci}
Let $p,q$ satisfy the conditions of Theorem~\ref{gwp} and $p>3$, our result yields the global solution of \eqref{CNS}
 for highly oscillatory initial velocity.  
 Let \( \phi \) be a Schwartz function whose Fourier transform \( \hat{\phi} \) is compactly supported, i.e., there exists \( M > 0 \) such that \( \hat{\phi}(\xi) = 0 \) for all \( |\xi| \geq M \). Define
\EQ{
\varphi_\varepsilon(x) = \cos\left(\frac{x_3}{\varepsilon}\right) \phi(x),
}
and set
\EQ{
u_{0,\varepsilon}(x) = (\partial_2 \varphi_\varepsilon(x), -\partial_1 \varphi_\varepsilon(x), 0), \quad a_{0,\varepsilon}(x) = 0.
}
Since \( \hat{\phi} \) is supported in \( |\xi| \leq M \), the support of \( \hat{\varphi_\varepsilon} \) lies in the region where \( |(\xi_1, \xi_2, \xi_3 \pm 1/\varepsilon)| \leq M \), which implies \( |\xi| \geq 1/\varepsilon - M \). Fix a low-frequency cutoff level \( J_0 \) and let \( R = 2^{J_0+1} \). For \( \varepsilon < 1/(R + M) \), we have \( 1/\varepsilon - M > R \), so \( \widehat{u_{0,\varepsilon}}(\xi) = 0 \) for all \( |\xi| \leq R \). Consequently, for all \( j \leq J_0 \), the Littlewood-Paley projections satisfy \(  \dot{\Delta}_{j} u_{0,\varepsilon} = 0 \), and thus the low-frequency Besov norm vanishes, that is
\[
\| u_{0,\varepsilon}^\ell \|_{\dot{B}_{q,1}^{-3+7/q}}  = 0.
\]
For the high-frequency part, since the Fourier transform of \( u_{0,\varepsilon} \) is concentrated near \( |\xi| \sim 1/\varepsilon \), we have (after normalizing \( \phi \) appropriately)
\[
\| u_{0,\varepsilon}^h \|_{\dot{B}_{p,1}^{-1+3/p}} \sim  \varepsilon^{1-3/p}.
\]
Therefore, for  \( 3 < p  \), and for sufficiently small \( \varepsilon \), the initial data  satisfies
\[
\| (a_{0,\varepsilon}, u_{0,\varepsilon})^\ell \|_{\dot{B}_{q,1}^{-3+7/q}} + \| a_{0,\varepsilon}^h \|_{\dot{B}_{p,1}^{3/p}} + \| u_{0,\varepsilon}^h \|_{\dot{B}_{p,1}^{-1+3/p}} \lesssim \varepsilon^{1-3/p},
\]
and hence generates a unique global solution to \eqref{CNS}.
\end{remark}

In the rest of the introduction, we would like to discuss the main ingredients in the proof of Theorem \ref{gwp}. Following \cite{FD2010,CMZ2010}, we consider the linearized compressible Navier-Stokes equation
\begin{equation}\label{eq:linearCNS}
\left\{
\begin{aligned}
& \partial_t a+\Div u=f,\\
&  \partial_{t}u-\mathcal{A}u+ \nabla a=g.
\end{aligned}
\right.
\end{equation}
For local well-posedness, the linear terms $\Div u$ and $\nabla a$ can be considered as perturbations and then merged to nonlinearity. However, we couldn't treat them as global-in-time perturbations. 
Applying the Fourier transform yields the explicit solution $(a,u)$ to \eqref{eq:linearCNS}, we have
\begin{align}\label{linear solution}
\begin{pmatrix} a \\ u \end{pmatrix} 
&= G(t)\begin{pmatrix} a_{0} \\ u_{0} \end{pmatrix} 
+ \int_{0}^{t} G(t-s)\begin{pmatrix} f \\ g \end{pmatrix}  \mathrm{d}s,
\end{align}
where $G(t)$ is the Fourier multiplier operator defined by the symbol (cf.~\cite{FD2010,CMZ2010})
\begin{equation}\label{Gtsymbol}
\widehat{G}(\xi,t) \triangleq 
\begin{pmatrix}
\dfrac{\lambda_{+}e^{\lambda_{-}t} - \lambda_{-}e^{\lambda_{+}t}}{\lambda_{+}-\lambda_{-}} 
& i\dfrac{e^{\lambda_{-}t} - e^{\lambda_{+}t}}{\lambda_{+}-\lambda_{-}}\xi^\top \\[2ex]
-i\dfrac{e^{\lambda_{-}t} - e^{\lambda_{+}t}}{\lambda_{+}-\lambda_{-}}\xi 
& \dfrac{\lambda_{+}e^{\lambda_{+}t} - \lambda_{-}e^{\lambda_{-}t}}{\lambda_{+}-\lambda_{-}}\dfrac{\xi\xi^\top}{|\xi|^2}
\end{pmatrix}
\end{equation}
with eigenvalues
\EQ{
\lambda_{\pm}=\lambda_\pm(|\xi|) = -\frac{1}{2}|\xi|^2 \pm \frac{1}{2}\sqrt{|\xi|^4 - 4|\xi|^2}.
}
One can see that, for high-frequency $|\xi|>2$, the semigroup $e^{t\lambda_\pm(D)}$ is parabolic, and 
\EQ{
\lambda_+\sim -|\xi|^2,\quad \lambda_-\sim -1, \quad \mbox{as}\quad |\xi|\to \infty.
}
For low frequency $|\xi|<2$, the semigroup $e^{t\lambda_\pm(D)}$ is parabolic-dispersive,
\EQ{
\lambda_\pm(\rho)=-\frac{\rho^2}{2}(1\pm iS(\rho)), \qquad S(\rho)=\sqrt{\frac{4}{\rho^2}-1}. 
}
Note that in the transition frequency region $||\xi|-2|\ll 1$, $G(t)$ behaves essentially like $e^{Ct\Delta}$.
The dispersive part affects the property of the semigroup $e^{t\lambda_\pm(D)}$ at low frequency. In the $L^2$-framework, by the Plancherel equality it has little impacts. All the previous works \cite{D2000,FD2010,CMZ2010,H20112} used $L^2$-based Besov spaces for low frequency. The energy method used in \cite{D2000} also relies on the $L^2$-based spaces to achieve some cancellations. However, the trade-off is in handling $high\times high \to low$ nonlinear estimates:
\EQ{\nonumber
\norm{P_{low}(f_{high}\cdot g_{high})}_{L^2}\les \norm{f_{high}}_{L^p}\cdot \norm{g_{high}}_{L^p}
}
which results in the restriction $p\leq 4$ for high-frequency. This is exactly the reason why this restriction is needed in the previous works. Our ideas compromise 

{\bf Idea 1.} Using $L^q$-framework for the low frequency.

In Section 3, we derive some estimates for general parabolic-dispersive semigroups. We tracked the impact of the dispersive part. In particular, we found that for low frequency
\EQ{
e^{t\lambda_\pm (D)}\approx e^{ct\Delta+iCtD}
}
Using the estimate for the wave operator $e^{itD}$, we prove for $1\leq q\leq \infty, k\leq 0$
\EQ{
\|e^{t\lambda_\pm (D)}\dot{P}_{k}f\|_{L^{q}(\R^d)}\lesssim e^{-ct 2^{2k}}2^{-(d-1)k|\frac{1}{2}-\frac{1}{q}|}\|\dot{P}_{k}f\|_{L^{q}(\R^d)},
}
from which we derive low frequency estimates for \eqref{eq:linearCNS} in $d=3$
\begin{equation}\label{eq:lowfrequency estimate}
\begin{aligned}
&\|(a,u) ^{\ell}\|_{\tilde{L}_{t}^{\infty}(\dot{B}_{q,1}^{s+|1-\frac{2}{q}|})\cap \tilde L_{t}^{1}(\dot{B}_{q,1}^{s+|1-\frac{2}{q}|+2})}
\les \|(a_{0},u_{0})^{\ell}
\|_{\dot{B}_{q,1}^{s}}+\|(f,g)^{\ell}\|_{\tilde L_{t}^{1}(\dot{B}_{q,1
}^{s})}.
\end{aligned}
\end{equation}
When $q=2$, it returns to the classical estimates. When $q>2$, it loses regularities on low-frequency. This loss will also affect $high\times high \to low$ nonlinear estimates. If only using \eqref{eq:lowfrequency estimate} and taking $q=p/2$, we will require $p<5$ eventually. To obtain the full range $p<6$, it requires our second idea. 

{\bf Idea 2.} Momentum formulation for low frequency.

Our ideas are inspired by the classical works of \cite{HoZu1995,HoZu1997}. To resolve the loss of regularities on low frequency, we use the momentum formulation. 
We define the momentum $m=\rho u$. Then the system for $(a,m)$ reads
\begin{equation*}
\left\{
\begin{aligned}
& \partial_t a + \mathrm{div}\, m = 0, \\
& \partial_t m - \mathcal{A} m +  \nabla a = h(a, m),
\end{aligned}
\right.
\end{equation*}
where $h=\nabla (\tilde h)$ has some null structures and hence improves $high\times high \to low$ interaction. This explains why we require that the initial momentum satisfies an additional low frequency condition in Theorem \ref{gwp}. 

We only use the momentum formulation for low frequency, as for high frequency we still use the effective velocity methods as in \cite{H20111, H20112}. Roughly speaking, we perform a nonlinear transform
\EQ{
(a,u)\to J(a,u):=(a,\tilde u)
}
where $\tilde u:=[(a+1)u]^{\ell}+u^h=(au)^{\ell}+u$, and do analysis for $(a,\tilde u)$. We can show the transform $J$ and its inverse is continuous in the resolution space $X_{q,p}$ (see Proposition~\ref{momentum}) under some smallness conditions. 

\begin{remark}
Our methods also work for higher dimensions and could improve the previous results. However, the extra loss of regularity index is $(d-1)(\frac{1}{2}-\frac{1}{q})$ in \eqref{eq:lowfrequency estimate} which increases as dimension increases, while the gain of regularity remains the same in the momentum formulation. It seems to us that new ideas are needed to obtain the optimal range $p<2d$. We hope to address this in a forthcoming work. 
\end{remark}

\section{Preliminaries} \label{mainestimate}

In this section, we collect some notations and present several lemmas that will be utilized subsequently. The Fourier transform of $f \in \mathcal{S}$ (the Schwartz class) is defined as
\[
\widehat{f}(\xi) = \mathcal{F}[f](\xi) :=(2\pi)^{-d/2} \int_{\mathbb{R}^{d}} f(x) e^{-i \xi \cdot x}  dx.
\]
$A \lesssim B$ means that $A \leq CB$ for some constant $C$. $A\sim B$ means $A\les B$ and $B\les A$.

Choose a real-valued $C_0^\infty(\R^n)$ function $\eta(\xi)$ such that: 
\begin{itemize}
    \item[(1)] $\eta$ is a radially symmetric and radially decreasing;
    \item[(2)] $\supp \eta\subset \{\xi\in \R^n:|\xi|\leq 1.01\}$; 
    \item[(3)] $0\leq \eta(\xi)\leq
1$, $\eta(\xi)\equiv 1$ if $|\xi|\leq 1$.
\end{itemize}
Let $\psi(\xi)=\eta(\xi/2)-\eta(\xi)$. For $k\in \Z$, we define
$\psi_k(\xi)=\psi(2^{-k}\xi)$ and
\EQ{\nonumber \eta_k(\xi)=
\CAS{
\psi_k(\xi), &\mbox{ if } k\geq 1;\\
\eta(\xi), &\mbox{ if } k=0;\\
0, &\mbox{ if } k\leq -1.
}
}
Then we define the Littlewood-Paley projectors:
for $k\in \Z$
\begin{align*}
\wh{\dot{P}_kf}(\xi)=\psi_k(\xi)\hat{f},\,\wh{{P}_kf}(\xi)=\eta_k(\xi)\hat{f},
\, \wh{{P}_{\leq k}f}(\xi)=\eta(2^{-k}\xi)\hat{f}, \, P_{\geq k}=I-P_{\leq k-1}.
\end{align*}
We also write $\dot \Delta_j=\dot P_j$, $\dot S_j=P_{\leq j-1}$.

Throughout this paper, for \( z \in \mathcal{S}'(\mathbb{R}^d) \), we use $z^{\ell}$, $z^h$ to denote the low-frequency and high-frequency component respectively, namely
\begin{equation}\label{hilo}
z^{\ell}:=\dot{S}_{k_{0}}z \quad \text{and} \quad 
z^{h}:=z-z^{\ell},
\end{equation}
where $k_0$ is a universal parameter to be determined later. $(a,u)^{\ell}$ (resp $(a,u)^h$) means $(a^{\ell},u^{\ell})$ (resp $(a^h,u^h)$).

Denote by $\mathcal{S}'_{0}:=\mathcal{S'}/\mathcal{P}$ the tempered distributions modulo polynomials $\mathcal{P}$. For $s\in \mathbb{R}$ and $1\leq p,r\leq \infty$, the homogeneous Besov spaces $\dot{B}^s_{p,r}$ are defined by
$$\dot{B}^s_{p,r}:=\Big\{f\in \mathcal{S}'_{0}:\|f\|_{\dot{B}^s_{p,r}}<\infty  \Big\} ,$$
where
\begin{equation*}
\|f\|_{\dot{B}^s_{p,r}}:=\Big(\sum_{q\in\mathbb{Z}}(2^{qs}\|\dot{\Delta}_qf\|_{L^{p}})^{r}\Big)^{1/r}
\end{equation*}
with the usual convention if $r=\infty$. We will also need the Hardy space $\Hl^1(\R^d)$ which is the Banach space with the following norm
\EQ{\label{Hardyspace}
\norm{f}_{\Hl^1}:=\normb{(\sum_k |\dot \Delta_k f|^2)^{1/2}}_{L^1}.
}
We often use the following classical properties of Besov spaces (see \cite{BaChDa11}).

$\bullet$ \ \emph{Interpolation:} The following inequality is satisfied for $1\leq p,r_{1},r_{2}, r\leq \infty, \sigma_{1}\neq \sigma_{2}$ and $\theta \in (0,1)$:
$$\|f\|_{\dot{B}_{p,r}^{\theta \sigma_{1}+(1-\theta )\sigma_{2}}}\lesssim \|f\| _{\dot{B}_{p,r_{1}}^{\sigma_{1}}}^{\theta} \|f\|_{\dot{B}_{p,r_2}^{\sigma_{2}}}^{1-\theta }$$
with $\frac{1}{r}=\frac{\theta}{r_{1}}+\frac{1-\theta}{r_{2}}$.

$\bullet$ \ \emph{Embedding:} For any $p\in[1,\infty]$ we have the continuous embedding $\dot {B}^{0}_{p,1}\hookrightarrow L^{p}\hookrightarrow \dot{B}^{0}_{p,\infty}$.

$\bullet$ \ If $\sigma\in \R$, $1\leq p_{1}\leq p_{2}\leq\infty$ and $1\leq r_{1}\leq r_{2}\leq\infty,$ then $\dot {B}^{\sigma}_{p_1,r_1}\hookrightarrow
\dot {B}^{\sigma-d\,(\frac{1}{p_{1}}-\frac{1}{p_{2}})}_{p_{2},r_{2}}$.

$\bullet$ \ The space $\dot {B}^{\frac {d}{p}}_{p,1}$ is continuously embedded in the set of bounded  continuous functions (going to zero at infinity if, additionally, $p<\infty$).

$\bullet$ \ \emph{Bernstein inequality:} 
\begin{equation*}\label{Eq:2.6}
\|D^{k}f\|_{L^{b}}
\les \lambda^{k+d(\frac{1}{a}-\frac{1}{b})}\|f\|_{L^{a}}
\end{equation*}
holds for all functions $f$ such that $\mathrm{Supp}\,\mathcal{F}f\subset\left\{\xi\in \R^{d}: |\xi|\leq \lambda\right\}$ for $\lambda>0$, if $k\in\mathbb{N}$ and $1\leq a\leq b\leq\infty$.

We denote the function spaces ${L^{q}_{T}(\dot{B}_{p,r}^{s})}$ and ${\bar{L}^{q}_{T}(\dot{B}_{p,r}^{s})}$ to be Banach spaces with the following norms:
\EQN{
\|u(t,\cdot)\|_{L^{q}_{T}(\dot{B}_{p,r}^{s})}=&\|\|u(t,\cdot)\|_{\dot{B}_{p,r}^{s}}
\|_{L^{q}_{T}},\\
\|u(t,\cdot)\|_{\bar{L}^{q}_{T}(\dot{B}_{p,r}^{s})} =& \big\| \big( 2^{js} \|\dot \Delta_{j}u(t,\cdot)\|_{L_{T}^{q}(L^{p}) \big) \big\|_{\ell^{r}(\mathbb{Z})}.
}}
Note that by Minkowski's inequality, for $r \leq q$ we have
\[
\|u(t,\cdot)\|_{L^{q}_{T}(\dot{B}_{p,r}^{s})} \leq \|u(t,\cdot)\|_{\tilde{L}^{q}_{T}(\dot{B}_{p,r}^{s})},
\]
with equality if and only if $q = r$. The opposite inequality holds when $r \geq q$.
The index $T$ will be omitted when $T = +\infty$. We denote by $\mathcal{C}_{b}(\dot{B}_{p,r}^{s})$ the subset of functions in ${L}^{\infty}(\dot{B}_{p,r}^{s})$ that are also continuous from $\mathbb{R}_{+}$ to $\dot{B}_{p,r}^{s}$. 

Product estimates in Besov spaces play a fundamental role in bounding bilinear terms. We first recall the well-known Bony's decomposition: for any distributions $f$ and $g$,
\[
fg = {T}_{f}g + {R}(f,g) + {T}_{g}f,
\]
where the paraproduct ${T}_f g$ and remainder ${R}(f,g)$ are defined as follows:
\[
{T}_{f}g \triangleq \sum_{j' \in \mathbb{Z}} \dot{S}_{j'-1}f \, \dot{\Delta}_{j'}g, \qquad 
{R}(f,g) \triangleq \sum_{j' \in \mathbb{Z}} \tilde{\dot{\Delta}}_{j'}f \, \dot{\Delta}_{j'}g.
\]
Here, $\tilde{\dot{\Delta}}_{j'} \triangleq \sum_{|j' - k| \leq 1} \dot{\Delta}_{k}$.

\begin{lemma}[Proposition 1.4 and Proposition 1.6, \cite{D2005}]\label{prodlemma}
Let \( s > 0 \) and \( 1 \leq p, r \leq \infty \). Then the space \( \dot{B}^{s}_{p,r} \cap L^{\infty} \) is an algebra, and we have the estimate
\[
\| f g \|_{\dot{B}^{s}_{p,r}} \lesssim \| f \|_{L^{\infty}} \| g \|_{\dot{B}^{s}_{p,r}} + \| g \|_{L^{\infty}} \| f \|_{\dot{B}^{s}_{p,r}}.
\]
If \( s_1, s_2 \leq \frac{d}{p} \) with \( s_1 + s_2 > d \max\left\{0, \frac{2}{p} - 1\right\} \), then
\[
\| a b \|_{\dot{B}^{s_1 + s_2 - \frac{d}{p}}_{p,1}} \lesssim \| a \|_{\dot{B}^{s_1}_{p,1}} \| b \|_{\dot{B}^{s_2}_{p,1}}.
\]
If \( s_1 \leq \frac{d}{p} \), \( s_2 < \frac{d}{p} \), and \( s_1 + s_2 \geq d \max\left\{0, \frac{2}{p} - 1\right\} \), then
\[
\| a b \|_{\dot{B}^{s_1 + s_2 - \frac{d}{p}}_{p,\infty}} \lesssim \| a \|_{\dot{B}^{s_1}_{p,1}} \| b \|_{\dot{B}^{s_2}_{p,\infty}}.
\]
\end{lemma}

\begin{lemma}[Proposition A.3, \cite{DX2017}]
\label{linestimate} Let \( F : \mathbb{R} \to \mathbb{R} \) be a smooth function with \( F(0) = 0 \). For all \( 1 \leq p, r \leq \infty \) and \( \sigma > 0 \), if \( f \in \dot{B}_{p,r}^{\sigma} \cap L^\infty \), then \( F(f) \in \dot{B}_{p,r}^{\sigma} \cap L^\infty \), and 
\begin{equation*}
\|F(f)\|_{\dot{B}_{p,r}^{\sigma}} \leq C \|f\|_{\dot{B}_{p,r}^{\sigma}}
\end{equation*}
where the constant \( C \) depends only on \(\|f\|_{L^\infty}\), \(F'\) (and higher-order derivatives), 
\(\sigma\), \(p\), and the dimension \(d\).  
\noindent If \(\sigma > -\min\left(\frac{d}{p}, \frac{d}{p'}\right)\) (where \(p'\) is the conjugate exponent satisfying \(1/p + 1/p' = 1\)), then \(f \in \dot{B}_{p,r}^{\sigma} \cap \dot{B}_{p,1}^{\frac{d}{p}}\) implies \(F(f) \in \dot{B}_{p,r}^{\sigma} \cap \dot{B}_{p,1}^{\frac{d}{p}}\), and 
\begin{equation*}
\|F(f)\|_{\dot{B}_{p,r}^{\sigma}} \leq C(1 +
\|f\|_{\dot{B}_{p,1}^{\frac{d}{p}}}) \|f\|_{\dot{B}_{p,r}^{\sigma}}.
\end{equation*}
\end{lemma}
\begin{remark}
To prove our main results when \( d \geq 2 \) and \( 1 < p < 2d \), we will use the above lemma with \( \sigma = \frac{d}{p} \) or \( \sigma = \frac{d}{p} - 1 \).
\end{remark}

\begin{lemma}[Proposition 2.1, \cite{D2014}]\label{heatm}
Let \(T>0\), \(\mu>0\), \(2\mu+\lambda>0\), \(s\in\mathbb{R}\), \(1\leq p\leq\infty\) and \(1\leq\varrho_{2}\leq\varrho_{1}\leq\infty\). 
Assume that \(u_{0}\in\dot{B}_{p,1}^{s}\) and \(f\in\tilde{L}^{\varrho_{2}}(0,T;\dot{B}_{p,1}^{s-2+\frac{2}{\varrho_{2}}})\) hold. If \(u\) is a solution of
\case{
\partial_{t}u-\mu\Delta u-(\mu+\lambda)\nabla\operatorname{div}u=&f, \quad (x,t)\in\mathbb{R}^{d}\times (0,T), \\
    u(x,0)=&u_{0}
}
    then \(u\) satisfies
    \[
    \min\{\mu,2\mu+\lambda\}^{\frac{1}{\varrho_{1}}}
    \|u\|_{\tilde{L}_{T}^{\varrho_{1}}(\dot{B}_{p,1}^{s+\frac{2}{\varrho_{1}}})} 
    \lesssim \|u_{0}\|_{\dot{B}_{p,1}^{s}} +
     \|f\|_{\tilde{L}_{T}^{\varrho_{2}}(\dot{B}_{p,1}^{s-2+\frac{2}{\varrho_{2}}})}.
    \]
\end{lemma}

\begin{lemma}[Theorem 2.2, \cite{RH2015}]\label{transport}
Let $1\leq p\leq p_{1}\leq \infty$,$1\leq r\leq \infty$ and $s\in\mathbb{R}$. Assume that
\begin{align*}
-\min(\frac{d}{p_1}, \frac{d}{p'}) <s<1+\frac{d}{p_{1}}.
\end{align*}
Then any smooth enough solutions of
\case{\label{eq:TDep}
\partial_t a + v\cdot \nabla a + \lambda a =& F, \\
a(0) =& a_0,
}
where $\lambda\geq 0$, satisfy
\begin{equation*}
\begin{aligned}
\|a\|_{\tilde{L}_{t}^{\infty}\dot{B}_{p,r}^{s}}+
\lambda\|a\|_{\tilde{L}_{t}^{1}\dot{B}_{p,r}^{s}}
\leq Ce^{CV(t)}(\|a_{0}\|_{\dot{B}_{p,r}^{s}}+\|F\|_{\tilde{L}_{t}^{1}\dot{B}_{p,r}^{s}})
\end{aligned}
\end{equation*}
with
\begin{equation*}\label{2.30}
\begin{aligned}
V(t)=\int_{0}^{t}\|\nabla v(\tau)\|_{\dot{B}^{\frac{d}{p_1}}_{p_1,\infty}\cap L^\infty}d\tau.
\end{aligned}
\end{equation*}
In the case $s=1+\frac{d}{p_{1}}$ and $r=1$, the above inequality is true with 
$V^{\prime}(t)=\|\nabla v(\tau)\|_{\dot{B}^{\frac{d}{p_{1}}}_{p_{1},\infty}}$.
\end{lemma}

\section{Estimates on parabolic-dispersive semigroup}
For many models in fluid dynamics, the low frequency components satisfy a parabolic-dispersive equation. In this section, we prove some estimates for parabolic-dispersive semigroup. We consider
\EQ{\nonumber
e^{-th(D)}f=\mathcal{F}^{-1}(e^{-th(\xi)}\hat{f})
}
and assume $h$ is of parabolic-dispersive type
\EQ{\nonumber h(\xi)=|\xi|^{a}+i\omega(\xi), \quad \xi\in \R^d.}
\begin{lemma}\label{le-e1}
Let  $f_{k}(x)=f(\frac{x}{2^{k}})$. Then we have
\begin{equation*}
\begin{aligned}
\|e^{-th(D)}\dot{\Delta}_{k}f\|_{L^{p}}=2^{-\frac{kd}{p}}\|e^{-th(2^kD)}\dot{\Delta}_{0}f_{k}\|_{L^{p}}.
\end{aligned}
\end{equation*}
\end{lemma}
\begin{proof}
By the properties of Fourier transform, we have
\begin{equation}\label{95}
\begin{aligned}
e^{-th(D)}\dot{\Delta}_{k}f=&\mathcal{F}^{-1}(e^{-th(\xi)}\hat{f}(\xi)\chi(\frac{\xi}{2^{k}}))\\
=&\mathcal{F}^{-1}(e^{-th(2^{k}\frac{\xi}{2^{k}})}\hat{f}(2^{k}\frac{\xi}{2^{k}})
\chi(\frac{\xi}{2^{k}}))\\
=&2^{kn}[\mathcal{F}^{-1}(e^{-th(2^{k}\xi)}\hat{f}(2^{k}\xi)
\chi(\xi))](2^{k}x)\\
=&[e^{-th(2^{k}D)}\dot{\Delta}_{0}f_{k}](2^{k}x).\\
\end{aligned}
\end{equation}  
Then by change of variables we complete the proof of the lemma. 
\end{proof}
If the dispersion is of equal or higher order compared to the parabolic part, then dispersion does not affect the estimates on the semigroup.  More precisely, we have

\begin{lemma}\label{le-e2}
Assume $h(\xi)=|\xi|^{a}+i\omega(\xi)$ with some $a>0$. Assume $k\leq 0$ and for $|\xi|\sim 2^k$, $|\partial_\xi^j\omega(\xi)|\lesssim 2^{k(a-j)}$ for $|j|=0,1,2, \cdots, [\frac{d}{2}]+1.$
Then there exist $0<c, C<\infty$ such that for $1\leq p\leq \infty$
\begin{equation*}
\begin{aligned}
e^{-Ct 2^{ak}}\|\dot{\Delta}_{k}f\|_{L^{p}}\lesssim \|e^{-th(D)}\dot{\Delta}_{k}f\|_{L^{p}}\lesssim  e^{-ct 2^{ak}}\|\dot{\Delta}_{k}f\|_{L^{p}}.
\end{aligned}
\end{equation*}
\end{lemma}
\begin{proof}
We first prove the upper bound. Since $\|\dot{\Delta}_{0}f_{k}\|_{L^{p}} = 2^{\frac{kd}{p}} \|\dot{\Delta}_{k}f\|_{L^{p}}$, by Lemma \ref{le-e1}, it suffices to show that
\[
\|e^{-th(2^kD)}\dot{\Delta}_{0}f_{k}\|_{L^{p}}=\|e^{-th(2^kD)}\tilde{\dot{\Delta}}_{0}\dot{\Delta}_{0}f_k\|_{L^{p}} \lesssim e^{-ct2^{ak}} \|\dot{\Delta}_{0}f_{k}\|_{L^{p}}.
\]
Note that $e^{-th(2^kD)}\tilde{\dot{\Delta}}_{0}$ is a convolution operator with kernel
$K_t(x) = \mathcal{F}^{-1}\left( e^{-th(2^k\xi)} \tilde{\psi}(\xi) \right)$,
where $\tilde{\psi}$ is the multiplier for $\tilde{\dot{\Delta}}_{0}$. It suffices to prove $\|K_t\|_{L^1}\les e^{-ct2^{ak}}$.

Let $m(\xi) = e^{-th(2^k\xi)}\tilde{\psi} (\xi) = e^{-t(2^{ak}|\xi|^a + i\omega(2^k\xi))} \tilde \psi(\xi)$. We estimate $\|K_t\|_{L^1}$ using the Bernstein multiplier estimate: for $L=[\frac{d}{2}]+1$
$$
\|K_t\|_{L^1} \lesssim \|m\|_{L^2}+\sum_{i=1}^{d} \| \partial_i^L m \|_{L^2}.
$$
Obviously, 
$\|m\|_{L^2} \lesssim e^{-ct2^{ak}}$.
It remains to estimate $\|\partial_i^Lm\|_{L^2}$. Let $\phi(\xi) = 2^{ak}|\xi|^a + i\omega(2^k\xi)$.  The assumption on $\omega$ implies: for $|\xi|\sim 1$, $|\partial^\beta \phi(\xi)| \lesssim 2^{ak}$ for any $\beta$. Therefore, we have
$|\partial_i^L m(\xi)| \lesssim (1 + t2^{ak})^L e^{-c't2^{ak}}\cdot 1_{|\xi|\sim 1}$ and thus $\|\partial_i^L m \|_{L^2} \lesssim e^{-ct2^{ak}}$.

On the other hands, for the lower bound, we have
\begin{equation}\label{A1.1}
\begin{aligned}
\|\dot{\Delta}_{k}f\|_{L^{p}}
\lesssim&\|e^{th(D)}e^{-th(D)}\dot{\Delta}_{k}f\|_{L^{p}}\\
\lesssim&\|\sum_{m=0}^{+\infty}\frac{(th(D))^{m}}{m!}e^{-th(D)}\dot{\Delta}_{k}f\|_{L^{p}}\\
\lesssim&\sum_{m=0}^{+\infty}\frac{C^{m}t^{m}}{m!}2^{akm}\|e^{-th(D)}\dot{\Delta}_{k}f\|_{L^{p}}\lesssim
 e^{Ct 2^{ak}}\|e^{-th(D)}\dot{\Delta}_{k}f\|_{L^{p}}.
\end{aligned}
\end{equation}
This proves the first inequality.
\end{proof}

In our application for compressible Navier-Stokes equations, we need to take
\begin{equation*}
\begin{aligned}
&h(\xi)=-\lambda_\pm(|\xi|)=|\xi|^{2}\mp i|\xi|\sqrt{4-|\xi|^{2}}.
\end{aligned}
\end{equation*}
One can see that $h$ does not satisfy the conditions in Lemma \ref{le-e2}, as the dispersion is of lower order. So we write $h$ as
\EQ{\label{hdef}
h(\xi)=(|\xi|^{2}\mp i|\xi|\sqrt{4-|\xi|^{2}}\pm 2i|\xi|)\mp i2|\xi|:=\widetilde h(\xi)\mp i2|\xi|.
}
It's easy to see that $\tilde h$ satisfies the conditions in Lemma \ref{le-e2}. So by this lemma we get for $1\leq p\leq \infty$ and $k\leq 0$
\EQ{\label{l1estimate} 
\|e^{t\lambda_\pm (D)}\dot{\Delta}_{k}f\|_{L^{p}}\lesssim  e^{-ct 2^{2k}}\norm{e^{\mp 2it D}\dot \Delta_k f}_{L^p}.
}
It reduces to an estimate for the wave operator $e^{itD}$.
We will need the following results due to Peral \cite{Peral} (Section IV) and Miyachi \cite{Miyachi}.

\begin{lemma}[Theorem 4.2, \cite{Miyachi}]
Let $T_mf=\ft^{-1}[1-\eta(\xi)]|\xi|^{-b}e^{i|\xi|} \ft f$. Then
\EQ{
\norm{T_mf}_{\Hl^1(\R^d)}\les \norm{f}_{\Hl^1(\R^d)}
}
if and only if $b\geq \frac{d-1}{2}$. Here $\Hl^1$ is the Hardy space defined in \eqref{Hardyspace}. 

As a consequence, we get for $k\geq 0$
\EQ{\label{eq:L1wave}
\norm{e^{iD}\dot \Delta_k f}_{L^1(\R^d)}\les 2^{k(d-1)/2}\norm{f}_{L^1(\R^d)},\\
\norm{e^{i2^kD}\dot \Delta_0 f}_{L^1(\R^d)}\les 2^{k(d-1)/2}\norm{f}_{L^1(\R^d)}.
}
\end{lemma}

\begin{lemma}[NS-low frequency semigroup]\label{NS-low frequency semigroup}
Let $h=-\lambda_\pm$ be given by \eqref{hdef}. 
Assume $1\leq p\leq\infty$, $k\leq 0$. Then we have
\begin{equation}\label{lowfresemi}
\begin{aligned}
\|e^{-th (D)}\dot{\Delta}_{k}f\|_{L^{p}}\lesssim e^{-ct 2^{2k}}2^{-(d-1)k|\frac{1}{2}-\frac{1}{p}|}\|\dot{\Delta}_{k}f\|_{L^{p}}.
\end{aligned}
\end{equation}
\end{lemma}

\begin{proof}
We only prove for $h=-\lambda_+=|\xi|^{2}- i|\xi|\sqrt{4-|\xi|^{2}}$ as the other case is similar. 
By Plancherel's equality, we get \eqref{lowfresemi} holds for $p=2$.  
By the Riesz-Thorin interpolation and then duality, it suffices to prove \eqref{lowfresemi} for $p=1$.

By \eqref{l1estimate} we get
\EQ{
\|e^{t\lambda_+ (D)}\dot{\Delta}_{k}f\|_{L^{1}}\les&  e^{-ct 2^{2k}}\norm{e^{-2it D}\dot \Delta_k f}_{L^1}\\
\les& e^{-ct 2^{2k}} 2^{-kd} \|e^{-i2^kt|D|} \dot{\Delta}_0 f_k\|_{L^1}
}
where $f_k(x)=f(x/2^k)$. If $|t2^k|\les 1$, then by Taylor expansion, we have
\EQ{
2^{-kd}\|e^{-i2^kt|D|} \dot{\Delta}_0 f_k\|_{L^1}\les 2^{-kd}\sum_{m=0}^\infty \frac{\norm{(-i2^kt|D|)^m \dot\Delta_0 f_k}_{L^1}}{m!}\les 2^{-kd}\norm{f_k}_{L^1}=\norm{f}_{L^1}.
}
If $|t2^k|\gg 1$, then by \eqref{eq:L1wave}, we get
\EQ{
2^{-kd}\|e^{-i2^kt|D|} \dot{\Delta}_0 f_k\|_{L^1}\les |t2^k|^{(d-1)/2}\cdot 2^{-kd}\|f_k\|_{L^1}=|t2^k|^{(d-1)/2}\cdot \|f\|_{L^1}.
}
Therefore, we obtain
\[
\|e^{t \lambda_+(D)} \dot{\Delta}_k f\|_{L^1} \lesssim e^{-c t 2^{2k}} 2^{-\frac{(d-1)k}{2}} (1 + 2^{2k} t)^{\frac{d-1}{2}} \|\dot \Delta_k f\|_{L^1}\les e^{-c' t 2^{2k}} 2^{-\frac{(d-1)k}{2}} \|\dot \Delta_k f\|_{L^1}
\]
and hence complete the proof of Lemma \ref{NS-low frequency semigroup}.
\end{proof}

In particular, we can derive all the estimates for the low frequency components.
Our objective is to establish the following proposition using the low-frequency semigroup.

\begin{lemma}[Low frequency estimates] \label{low}
Assume that $k_{0}>0, \, s \in \mathbb{R}$, $p \in [1, \infty]$. Let $z = (z_{1}, z_{2})$ satisfy 
\begin{equation}\label{pabolic-dis}
\left\{
\begin{aligned}
& \partial_{t}z_{1}+\Div z_{2}=f\\
&  \partial_{t}z_{2}-\mathcal{A} z_{2}+\nabla z_{1}=g, \\
& (z_{1},z_{2})\big|_{t=0}=(z_{1,0},z_{2,0}). \\
\end{aligned}
\right.
\end{equation}
Then there exists a constant $C > 0$, depending only on $k_{0}$, such that for all $t \geq 0$,
\begin{equation}\label{prop-estimate}
\bigl\|(z_{1}, z_{2})^{\ell}\bigr\|_{\tilde{L}_{t}^{\rho_{1}} 
(\dot{B}_{p,1}^{s + |1 - \frac{2}{p}| + \frac{2}{\rho_{1}}})}
\leq C
\bigl\|(z_{10}, z_{20})^{\ell}\bigr\|_{\dot{B}_{p,1}^{s}} +C 
\bigl\|(f, g)^{\ell}\bigr\|_{\tilde{L}_{t}^{1}(\dot{B}_{p,1}^{s})}.
\end{equation}
\end{lemma}

\begin{proof}
Applying the Fourier transform yields the explicit solution $z$ to \eqref{pabolic-dis}, we have
\begin{align*}
z = \begin{pmatrix} z_1 \\ z_2 \end{pmatrix} 
&= G(t)\begin{pmatrix} z_{1,0} \\ z_{2,0} \end{pmatrix} 
+ \int_{0}^{t} G(t-s)\begin{pmatrix} f \\ g \end{pmatrix}  \mathrm{d}s,
\end{align*}
where $G(t)$ is given by \eqref{Gtsymbol}.

We denote $k_1=\max\{k: 2^{2k+2}-2^k<0\}$. Then direct computation shows that for $|\xi|\sim 2^k\leq 2^{k_1}$
\EQ{
|\nabla^\alpha_\xi (\frac{\xi}{\lambda_{+}-\lambda_{-}},\frac{\lambda_{\pm}}{\lambda_{+}-\lambda_{-}}
)|\lesssim |\xi|^{-\alpha},\quad|\alpha|\geq 0. 
}
Thus $\frac{\xi}{\lambda_{+}-\lambda_{-}}\psi_k(\xi), \frac{\lambda_{\pm}}{\lambda_{+}-\lambda_{-}}\psi_k(\xi)$ are $L^p$ multipliers for $1\leq p\leq \infty$ with uniform norms in $k$. By Lemma \ref{le-e2}, we   
have for $k\leq k_1$, we obtain
\begin{equation}\label{semi}
\begin{aligned}
\|(\dot{\Delta}_k\dot{S}_{k_0} z_1(t), \dot{\Delta}_k \dot{S}_{k_0}z_2(t))\|_{L^p} 
&\lesssim e^{-ct 2^{2k}} 2^{-2k\left|\frac{1}{2}-\frac{1}{p}\right|} 
\|(\dot{\Delta}_k \dot{S}_{k_0}z_{10}, \dot{\Delta}_k \dot{S}_{k_0}z_{20})\|_{L^p} \\
&\quad + \int_0^t e^{-c(t-s) 2^{2k}} 2^{-2k\left|\frac{1}{2}-\frac{1}{p}\right|} 
\|(\dot{\Delta}_k \dot{S}_{k_0}f(s), \dot{\Delta}_k \dot{S}_{k_0}g(s))\|_{L^p}  \mathrm{d}s.
\end{aligned}
\end{equation}
For $k_{1}< k\lesssim k_0$, due to the availability of both upper and lower bounds on the frequency variable, using the argument in Proposition 4.4 (b) of \cite{CMZ2010}, we can also establish inequality \eqref{semi}. 
 
Taking the $L^{\rho_1}$-norm in time for \eqref{semi} yields
\begin{equation*}
\begin{aligned}
&2^{2k|\frac{1}{2}-\frac{1}{p}|} \|(\dot{\Delta}_k 
\dot{S}_{k_0}z_1, \dot{\Delta}_k \dot{S}_{k_0}z_2)\|_{L_t^{\rho_1}(L^p)} \\
&\quad \lesssim \left( \frac{1 - e^{-c t \rho_1 
2^{2k}}}{c \rho_1 2^{2k}} \right)^{\!\!\frac{1}{\rho_1}} 
\|(\dot{\Delta}_k \dot{S}_{k_0}z_{10}, \dot{\Delta}_k \dot{S}_{k_0}z_{20})\|_{L^p} \\
&\quad + \left( \frac{1 - e^{-c t \rho_2 2^{2k}}}
{c \rho_2 2^{2k}} \right)^{\!\!\frac{1}{\rho_2}} \|(\dot{\Delta}_k \dot{S}_{k_0}f, \dot{\Delta}_k \dot{S}_{k_0}g)
\|_{L_t^\rho(L^p)},
\end{aligned}
\end{equation*}
where $\frac{1}{\rho_2} = 1 + \frac{1}{\rho_1} - \frac{1}{\rho}$. Multiplying by 
$2^{k(s + 2/\rho_1)}$ and taking the $\ell^1$-norm gives
\begin{equation*}
\begin{aligned}
&\|(z_1,z_2)^{\ell}\|_{\tilde{L}_t^{\rho_1}(\dot{B}_{p,1}^{s + |1 - 2/p| + 2/\rho_1})} \\
&\quad \lesssim \sum_{k} \left( \frac{1 - e^{-c t \rho_1 2^{2k}}}{c \rho_1} 
\right)^{\!\!\frac{1}{\rho_1}} 2^{k s} \|(\dot{\Delta}_k \dot{S}_{k_0}z_{10}, \dot{\Delta}_k \dot{S}_{k_0}z_{20})\|_{L^p} \\
&\quad + \sum_{k } \left( \frac{1 - e^{-c t \rho_2 2^{2k}}}
{c \rho_2} \right)^{\!\!\frac{1}{\rho_2}} 2^{k(s - 2 + 2/\rho)} 
\|(\dot{\Delta}_k \dot{S}_{k_0}f, \dot{\Delta}_k\dot{S}_{k_0} g)\|_{L_t^\rho(L^p)} \\
&\quad \lesssim \left( \frac{1 - e^{-c t \rho_1}}{c \rho_1} 
\right)^{\!\!\frac{1}{\rho_1}} \|(z_{10},z_{20})^{\ell}\|_{\dot{B}_{p,1}^{s}}+ 
\left( \frac{1 - e^{-c t \rho_2}}{c \rho_2} 
\right)^{\!\!1 + \frac{1}{\rho_1} - \frac{1}{\rho}} 
\|(f,g)^{\ell}\|_{\tilde{L}_t^\rho(\dot{B}_{p,1}^{s - 2 + 2/\rho})}.
\end{aligned}
\end{equation*}
In particular, 
\begin{equation*}
\begin{aligned}
\|(z_1,z_2)^{\ell}\|_{\tilde{L}_t^{\rho_1}(\dot{B}_{p,1}^{s + |1 - 2/p| +
 2/\rho_1})} \lesssim 
\|(z_{10},z_{20})^{\ell}\|_{\dot{B}_{p,1}^{s}} + \|(f,g)^{\ell}
\|_{\tilde{L}_t^1(\dot{B}_{p,1}^{s})}.
\end{aligned}
\end{equation*}
Then we establish the desired inequality.
\end{proof}

\section{Priori estimates}

We assume $(a,u)$ is a solution to \eqref{CNS}. In what follows, let us define
$\mathcal{Y} := \mathcal{Y}^\ell + \mathcal{Y}^h$,
where
\EQ{
\begin{aligned}
\mathcal{Y}^\ell :=& \|m^{\ell}\|_{\tilde{L}_{t}^{\infty}(\dot{B}_{q,1}^{-1 + \frac{3}{q}}) \cap \tilde{L}_{t}^{2}(\dot{B}_{q,1}^{-1 + \frac{5}{q}}) \cap \tilde{L}_{t}^{1}(\dot{B}_{q,1}^{\frac{5}{q}})},\\
\mathcal{Y}^h :=& \|m^h\|_{\tilde{L}_{t}^{\infty}(\dot{B}_{p,1}^{-1 + \frac{3}{p}}) \cap \tilde{L}_{t}^{1}(\dot{B}_{p,1}^{\frac{3}{p}})}
\end{aligned}
}
and
\EQ{
\begin{aligned}
&\mathcal{X} := \|(a,u)^{\ell}\|_{\tilde{L}_{t}^{\infty}(\dot{B}_{q,1}^{-1 + \frac{3}{q}}) \cap \tilde{L}_{t}^{2}(\dot{B}_{q,1}^{-1 + \frac{5}{q}}) \cap \tilde{L}_{t}^{1}(\dot{B}_{q,1}^{\frac{5}{q}})} \\
& \  \ \quad +\|a^{h}\|_{\tilde{L}_{t}^{\infty}(\dot{B}_{p,1}^{\frac{3}{p}}) \cap \tilde{L}_{t}^{1}(\dot{B}_{p,1}^{\frac{3}{p}})} + \|u^{h}\|_{\tilde{L}_{t}^{\infty}(\dot{B}_{p,1}^{-1 + \frac{3}{p}}) \cap \tilde{L}_{t}^{1}(\dot{B}_{p,1}^{1 + \frac{3}{p}})},\\
&\mathcal{X}_0:=\|(a_{0},m_{0})^{\ell}\|_{\dot{B}_{q,1}^{-3
+\frac{7}{q}}}+\|(a_{0},u_0)^h\|_{\X_p}.
\end{aligned}
}
Under the conditions of Theorem~\ref{gwp}, we easily obtain the following estimates
\begin{equation*}
\begin{aligned}
&\|m\|_{\tilde{L}_{t}^{\infty}(\dot{B}_{p,1}^{-1+\frac{3}{p}})}
\lesssim  \mathcal{Y},\ \ \|a\|_{\tilde{L}_{t}^{\infty}(\dot{B}_{p,1}^{-1+\frac{3}{p}})}
\lesssim \mathcal{X},\ \ \|u\|_{\tilde{L}_{t}^{\infty}(\dot{B}_{p,1}^{-1+\frac{3}{p}})}
\lesssim  \mathcal{X},\\
&\|u\|_{\tilde{L}_{t}^{1}(\dot{B}_{p,1}^{1+\frac{3}{p}})}
\lesssim \mathcal{X}, \  \ \|a\|_{\tilde{L}_{t}^{2}(\dot{B}_{p,1}^{\frac{3}{p}})}
\lesssim \mathcal{X}, \  \ \|u\|_{\tilde{L}_{t}^{2}(\dot{B}_{p,1}^{\frac{3}{p}})}
\lesssim  \mathcal{X}, \ \ \|a\|_{\tilde{L}_{t}^{\infty}(\dot{B}_{p,1}^{\frac{3}{p}})}
\lesssim\mathcal{X},
\end{aligned}
\end{equation*}
which we will use below.
The main purpose of this section is to establish the following uniform priori estimates:
\begin{prop}\label{priori}
Assume $2 \leq p< 6, \,2 \leq q \leq p \leq 2q$, and
 that $(q,\,\, p)$ satisfies
\begin{align*}
\frac{3}{q}-\frac{3}{p}\leq 1<\frac{2}{q}+\frac{3}{p}.
\end{align*}  Then we have 
\begin{equation*}
\begin{aligned}
\mathcal{X}\lesssim\mathcal{X}_0 +  \mathcal{X}^{2}.
\end{aligned}
\end{equation*}
\end{prop}

The above proposition follows from Propositions \ref{lowfre} and \ref{highfre} below, which concern low frequency and high frequency analysis respectively.

\subsection{The low frequency analysis}
We prove
\begin{prop}\label{lowfre}
Assume $2 \leq p< 6,\,\,2 \leq q \leq p \leq 2q$, and that $(q,\,\, p)$ satisfies
\begin{align*}
\frac{3}{q}-\frac{3}{p}\leq 1<\frac{2}{q}+\frac{3}{p},
\end{align*}  then we have the following low estimates
\begin{equation*}
\begin{aligned}
\|(a, u)^{\ell}\|_{\tilde{L}_{t}^{\infty}(\dot{B}_{q,1}^{-1
+\frac{3}{q}})\cap\tilde{L}_{t}^{2}(\dot{B}_{q,1}^{-1+\frac{5}{q}})
\cap\tilde{L}_{t}^{1}(\dot{B}_{q,1}^{\frac{5}{q}})} \lesssim\mathcal{X}_0 + \mathcal{X}^{2}.
\end{aligned}
\end{equation*}
\end{prop}
\begin{proof}
We will prove Proposition \ref{lowfre} using results from Subsection \eqref{Momentum estimates}, 
Subsection \eqref{linear part} and Subsection \eqref{nonlinear part}.
\end{proof}
In the low frequency analysis, due to the $high\times high\to low$ interaction, the loss of regularities is problematic in the $L^q$-framework. In order to overcome that, we apply the momentum formula which achieves some null structure in the nonlinear interactions. Specifically, for the momentum $m = \rho u$, then system for $(a, m)$ reads
\begin{equation}\label{comnscauchy}
\left\{
\begin{aligned}
& \partial_t a + \mathrm{div}\, m = 0, \\
& \partial_t m - \mathcal{A} m + \nabla a = h(a, m),
\end{aligned}
\right.
\end{equation}
where $\mathcal{A} m = \mu\Delta m + (\mu+ \lambda) \nabla \mathrm{div} m$, and $h(a, m)= \sum_{i=1}^{3} h_i(a, m)$ with
\EQ{
\left\{
\begin{aligned}
h_1(a, m) &= \mathrm{div}\bigl( (I(a) - 1) m \otimes m \bigr), \\[1mm]
h_2(a, m) &= -\mathcal{A} (I(a) m), \\[1mm]
h_3(a, m)&= \nabla(G(a)a).
\end{aligned} \right.
\label{nonlinear}
}
We use the following fact
\begin{equation*}
\begin{aligned}
&\nabla P(\rho)=\nabla(P'(1)a+G(a)a):=\nabla a+h_3.
\end{aligned}
\end{equation*}

\subsubsection{ \bf Momentum estimates}\label{Momentum estimates}

We first state the following proposition concerning the relationship between $\mathcal{X}$ and $\mathcal{Y}$.
\begin{prop}\label{momentum}
Let $2 \leq p< 6,\,\,2 \leq q \leq p \leq 2q$, and that $(q,\,\, p)$ satisfy
\begin{align*}
\frac{3}{q}-\frac{3}{p}\leq 1<\frac{2}{q}+\frac{3}{p}.
\end{align*} 
Suppose $(a, u)$ is a solution to system \eqref{CNS} with $\mathcal{X}\lesssim 1$ and 
$\|a\|_{\tilde{L}_{t}^{\infty}(\dot{B}_{p,1}^{3/p})} \lesssim 1$. 
Then we have
\begin{align*}
\mathcal{Y}\leq C\mathcal{X}.
\end{align*}
Moreover, the following estimate holds that
\begin{equation*}
\|u^{\ell}\|_{\tilde{L}_{t}^{\infty}(\dot{B}_{q,1}^{-1
+\frac{3}{q}})\cap\tilde{L}_{t}^{2}(\dot{B}_{q,1}^{-1+\frac{5}{q}})
\cap\tilde{L}_{t}^{1}(\dot{B}_{q,1}^{\frac{5}{q}})}
\leq  \mathcal{Y}^{\ell}+C\mathcal{X}^{2}.
\end{equation*}
\end{prop}
\begin{proof}
 Defining \(I(a) = \frac{a}{a+1}\), we establish the fundamental relations
\[
m = u+au, \quad u = m - I(a)m.
\]
Applying Lemma \ref{prodlemma}, we have
\EQ{\label{m1}
\begin{aligned}
\|m\|_{\tilde{L}_t^2(\dot{B}_{p,1}^{3/p})} 
\leq& \|au + u\|_{\tilde{L}_t^2(\dot{B}_{p,1}^{3/p})} \\
 \lesssim& \bigl(1 + \|a\|_{\tilde{L}_t^{\infty}(\dot{B}_{p,1}^{3/p})}\bigr)
 \|u\|_{\tilde{L}_t^2(\dot{B}_{p,1}^{3/p})} \\
 \lesssim& \mathcal{X}(1 + \mathcal{X})\leq C\mathcal{X}.
\end{aligned}
}

\noindent \textbf{Step 1:  To estimate \(\mathcal{Y} \leq C\mathcal{X}\).}  

\(\bullet\) \textbf{High-frequency estimate for} 
\(\|m^h\|_{\tilde{L}_t^{\infty}(\dot{B}_{p,1}^{-1+3/p}) \cap \tilde{L}_t^1(\dot{B}_{p,1}^{3/p})}\). 

Since \(m = (a+1)u\),  for \(2 \leq p < 6\), Lemma \ref{prodlemma} yields
\EQ{\label{mml}
\begin{aligned}
\mathcal{Y}^h 
\lesssim& \|u^h\|_{\tilde{L}_t^{\infty}(\dot{B}_{p,1}^{3/p-1})} + 
\|a\|_{\tilde{L}_t^{\infty}(\dot{B}_{p,1}^{3/p})} 
\|u\|_{\tilde{L}_t^{\infty}(\dot{B}_{p,1}^{3/p-1})} \\
&\quad + \|u^h\|_{\tilde{L}_t^1(\dot{B}_{p,1}^{3/p+1})} 
+ \|a\|_{\tilde{L}_t^2(\dot{B}_{p,1}^{3/p})} 
\|u\|_{\tilde{L}_t^2(\dot{B}_{p,1}^{3/p})}\\
\lesssim& \mathcal{X} + \mathcal{X}^2.
\end{aligned}
}

\(\bullet\) \textbf{Low-frequency estimate for} \(\|m^\ell\|_{\tilde{L}_t^{\infty}(\dot{B}_{q,1}^{-1+3/q}) \cap \tilde{L}_t^2(\dot{B}_{q,1}^{-1+5/q}) \cap \tilde{L}_t^1(\dot{B}_{q,1}^{5/q})}\).

Using Lemma~\ref{lem key} under the conditions \(2 \leq p < 6\), \(\frac{3}{q} - \frac{3}{p} \leq 1\), and  \(q \leq p \leq 2q\), we have
\EQ{\label{au1}
\begin{aligned}
\|(au)^\ell\|_{\tilde{L}_t^{\infty}(\dot{B}_{q,1}^{-1+3/q})} 
&\lesssim \|(T_u a + R(u,a))^\ell\|_{\tilde{L}_t^{\infty}(\dot{B}_{q,1}^{-1+3/q})} + 
\|(T_a u)^\ell\|_{\tilde{L}_t^{\infty}(\dot{B}_{q,1}^{-1+3/q})} \\
&\lesssim \|a\|_{\tilde{L}_t^{\infty}(\dot{B}_{p,1}^{3/p})} 
\|u\|_{\tilde{L}_t^{\infty}(\dot{B}_{p,1}^{-1+3/p})} 
+ \|(T_a u)^\ell\|_{\tilde{L}_t^{\infty}(\dot{B}_{q,1}^{-2+3/q})} \\
&\lesssim \|a\|_{\tilde{L}_t^{\infty}(\dot{B}_{p,1}^{3/p})} 
\|u\|_{\tilde{L}_t^{\infty}(\dot{B}_{p,1}^{-1+3/p})}+ \|a\|_{\tilde{L}_t^{\infty}(\dot{B}_{p,1}^{3/p-1})} 
\|u\|_{\tilde{L}_t^{\infty}(\dot{B}_{p,1}^{-1+3/p})}\\
&\lesssim \mathcal{X}^2
\end{aligned}
}
and
\begin{equation}\label{au2}
\begin{aligned}
\|(au)^\ell\|_{\tilde{L}_t^2(\dot{B}_{q,1}^{-1+5/q})} 
&\lesssim \|(R(a,u) + T_u a)^\ell + (T_a u)^\ell
\|_{\tilde{L}_t^2(\dot{B}_{q,1}^{-1+3/q})} \\
&\lesssim \|a\|_{\tilde{L}_t^2(\dot{B}_{p,1}^{3/p})} 
\|u\|_{\tilde{L}_t^{\infty}(\dot{B}_{p,1}^{-1+3/p})}+ \|u\|_{\tilde{L}_t^2(\dot{B}_{p,1}^{3/p})} 
\|a\|_{\tilde{L}_t^{\infty}(\dot{B}_{p,1}^{-1+3/p})}\\
&\lesssim \mathcal{X}^2.
\end{aligned}
\end{equation}
For \(q \leq p \leq 2q\)  and  \(\frac{3}{q} - \frac{3}{p} \leq 1\),  applying Lemma~\ref{lem key}, we have
\begin{equation}\label{sj1}
\begin{aligned}
\|(T_{a}u^{\ell})^\ell + (T_{u}a^{\ell})^\ell\|_{\tilde{L}_t^1(\dot{B}_{q,1}^{5/q})}\lesssim &\|a\|_{\tilde{L}_t^\infty(L^{\infty})} 
\|u^{\ell}\|_{\tilde{L}_t^{1}(\dot{B}_{q,1}^{5/q})} + 
\|u\|_{\tilde{L}_t^2(L^{\infty})} 
\|a^{\ell}\|_{\tilde{L}_t^{2}(\dot{B}_{q,1}^{5/q})} \\
\lesssim& \|a\|_{\tilde{L}_t^\infty(\dot{B}_{p,1}^{3/p})} 
\|u^{\ell}\|_{\tilde{L}_t^{1}(\dot{B}_{q,1}^{5/q})} + 
\|u\|_{\tilde{L}_t^2(\dot{B}_{p,1}^{3/p})} 
\|a^{\ell}\|_{\tilde{L}_t^{2}(\dot{B}_{q,1}^{-1+5/q})}\\
\lesssim& \mathcal{X}^2,
\end{aligned}
\end{equation}
\begin{equation}\label{sj2}
\begin{aligned}
\|(R(a,u))^\ell\|_{\tilde{L}_t^1(\dot{B}_{q,1}^{5/q})} 
\lesssim& \|(R(a,u))^\ell
\|_{\tilde{L}_t^1(\dot{B}_{q,1}^{3/q})}\\
\lesssim& \|a\|_{\tilde{L}_t^2(\dot{B}_{p,1}^{3/p})} 
\|u\|_{\tilde{L}_t^{2}(\dot{B}_{p,1}^{3/p})}\lesssim \mathcal{X}^2,
\end{aligned}
\end{equation}
\begin{equation}\label{sj3}
\begin{aligned}
\|(T_{a} u^h)^\ell\|_{\tilde{L}_t^1(\dot{B}_{q,1}^{5/q})} 
\lesssim& \|T_{a} u^h\|_{\tilde{L}_t^1(\dot{B}_{q,1}^{-1+3/q})} \\
\lesssim& \|a\|_{\tilde{L}_t^{\infty}(\dot{B}_{p,1}^{-1+3/p})}
\|u^h\|_{\tilde{L}_t^1(\dot{B}_{p,1}^{1+3/p})}\lesssim \mathcal{X}^2.
\end{aligned}
\end{equation}
Use the decomposition 
\[
(T_{u} a^h)^{\ell} = [S_{k_0}, T_{u}] a^h + T_{u} S_{k_0}a^h.
\]
We focus on the commutator, for \(2 \leq q \leq p \leq 2q\).  Lemma~\ref{commo} gives
\begin{equation} \label{sj4}
\begin{aligned}
&\|[S_{k_0}, T_{u}] a^h
\|_{\tilde{L}_{t}^{1}(\dot{B}_{q,1}^{\frac{5}{q}})}\lesssim \|\nabla u
\|_{\tilde{L}_{t}^{1}(\dot{B}_{p,1}^{\frac{3}{p}})}
\|a^h\|_{\tilde{L}_{t}^{1}(\dot{B}_{p,1}^{\frac{3}{p}+\frac{2}{q}-1})} \\
&\lesssim \| u
\|_{\tilde{L}_{t}^{1}(\dot{B}_{p,1}^{1+\frac{3}{p}})}
\|a^h\|_{\tilde{L}_{t}^{1}(\dot{B}_{p,1}^{\frac{3}{p}})}
\lesssim \mathcal{X}^2.
\end{aligned}
\end{equation}
Similar to estimate \((T_{u}a^{\ell})^\ell\), we have
\begin{equation} \label{sj5}
\begin{aligned}
\|T_{u} S_{k_0}a^h\|_{\tilde{L}_{t}^{1}(\dot{B}_{q,1}^{5/q})} \lesssim \mathcal{X}^2.
\end{aligned}
\end{equation}
Gathering \eqref{sj1}, \eqref{sj2}, \eqref{sj3}, \eqref{sj4} and \eqref{sj5}, we obtain
\begin{equation} \label{au3}
\|(au)^\ell\|_{\tilde{L}_t^1(\dot{B}_{q,1}^{5/q})} \lesssim \mathcal{X}^2.
\end{equation}
Combining \eqref{au1}, \eqref{au2}, and \eqref{au3}, we have
\begin{equation}\label{mh}
\begin{aligned}
\mathcal{Y}^\ell 
\lesssim& \|u^\ell\|_{\tilde{L}_{t}^{\infty}(\dot{B}_{q,1}^{-1 + 3/q}) \cap \tilde{L}_{t}^{2}(\dot{B}_{q,1}^{-1 + 5/q}) \cap \tilde{L}_{t}^{1}(\dot{B}_{q,1}^{5/q})}
\\
& + \|(au)^\ell\|_{\tilde{L}_{t}^{\infty}(\dot{B}_{q,1}^{-1 + 3/q}) \cap \tilde{L}_{t}^{2}(\dot{B}_{q,1}^{-1 + 5/q}) \cap \tilde{L}_{t}^{1}(\dot{B}_{q,1}^{5/q})}\\
\lesssim& \mathcal{X} + \mathcal{X}^{2}.
\end{aligned}
\end{equation}
Finally, \eqref{mml} and  \eqref{mh} imply
\begin{align}\label{qu}
\mathcal{Y} \lesssim \mathcal{X}(1 + \mathcal{X}) \leq C\mathcal{X}.
\end{align}

\textbf{Step 2: To estimate $\|u^{\ell}\|_{\tilde{L}_{t}^{\infty}(\dot{B}_{q,1}^{-1
+\frac{3}{q}})\cap\tilde{L}_{t}^{2}(\dot{B}_{q,1}^{-1+\frac{5}{q}})
\cap\tilde{L}_{t}^{1}(\dot{B}_{q,1}^{\frac{5}{q}})}
\leq  \mathcal{Y}^{\ell}+C\mathcal{X}^{2}$.}

$\bullet$ \textbf{Low-frequency estimate for $\|(I(a)m)^\ell\|_{\tilde{L}_t^{\infty}(\dot{B}_{q,1}^{-1+3/q})}$.}  

It suffices to consider quadratic interactions. For $\frac{3}{q} - \frac{3}{p}\leq 1,  q\leq p \leq 2q $ and $2\leq p<6 $,  By Lemma \ref{lem key} and
Lemma \ref{linestimate}, we have
\begin{equation}\label{mom1}
\begin{aligned}
&\|(T_m I(a) + R(m, I(a)))^\ell\|_{\tilde{L}_t^{\infty}(\dot{B}_{q,1}^{-1+3/q})} \\
&\lesssim\|m\|_{\tilde{L}_t^{\infty}(\dot{B}_{p,1}^{3/p-1})} 
\|I(a)\|_{\tilde{L}_t^{\infty}(\dot{B}_{p,1}^{3/p})} \\
&\lesssim \|m\|_{\tilde{L}_t^{\infty}(\dot{B}_{p,1}^{3/p-1})}(1+
\|a\|_{\tilde{L}_t^{\infty}(\dot{B}_{p,1}^{3/p})})
\|a\|_{\tilde{L}_t^{\infty}(\dot{B}_{p,1}^{3/p})}\lesssim \mathcal{X} \mathcal{Y}(1+ \mathcal{X}).
\end{aligned}
\end{equation}
For $\frac{3}{q} - \frac{3}{p}  \leq 1,\,  q\leq p \leq 2q$  and  $2\leq p<6 $, using Lemma \ref{linestimate} and Lemma \ref{lem key}, we obtain 
\EQ{\label{mom2}
&\|(T_{I(a)}m)^\ell\|_{\tilde{L}_t^{\infty}(\dot{B}_{q,1}^{-1+3/q})} 
\lesssim \|(T_{I(a)}m)^\ell\|_{\tilde{L}_t^{\infty}(\dot{B}_{q,1}^{-2+3/q})} \\
&\lesssim \|I(a)\|_{\tilde{L}_t^{\infty}(\dot{B}_{p,1}^{3/p-1})} \|m\|_{\tilde{L}_t^{\infty}(\dot{B}_{p,1}^{3/p-1})} \\
&\lesssim(1+
\|a\|_{\tilde{L}_t^{\infty}(\dot{B}_{p,1}^{3/p})})
\|a\|_{\tilde{L}_t^{\infty}(\dot{B}_{p,1}^{3/p-1})}
 \|m\|_{\tilde{L}_t^{\infty}(\dot{B}_{p,1}^{3/p-1})} \lesssim 
 \mathcal{X} \mathcal{Y}(1+ \mathcal{X}).
}
Combining \eqref{mom1} and \eqref{mom2} yields
\begin{align*}\label{conclusion1}
\|(I(a)m)^\ell\|_{\tilde{L}_t^{\infty}(\dot{B}_{q,1}^{-1+3/q})}
 \leq C \mathcal{X} \mathcal{Y}.
\end{align*}

$\bullet$ \textbf{Low-frequency estimate for $\|(I(a)m)^\ell\|_{\tilde{L}_t^2(\dot{B}_{q,1}^{-1+5/q}) \cap \tilde{L}_t^1(\dot{B}_{q,1}^{5/q})}$.}  

For $\frac{3}{q} - \frac{3}{p}  \leq 1,\,  q\leq p \leq 2q,\, \frac{3}{q} - \frac{3}{p}  \leq 1$, by Lemma \ref{linestimate} and Lemma \ref{lem key}, we obtain
\begin{equation}\label{s0}
\begin{aligned}
&\|(T_m I(a))^\ell\|_{\tilde{L}_t^2(\dot{B}_{q,1}^{-1+5/q})} 
\lesssim \|(T_m I(a))^\ell\|_{\tilde{L}_t^2(\dot{B}_{q,1}^{-1+3/q})}\\
&\leq \|m\|_{\tilde{L}_t^{\infty}(\dot{B}_{p,1}^{-1+3/p})} 
\|I(a)\|_{\tilde{L}_t^2(\dot{B}_{p,1}^{3/p})}  \\
&\leq \|m\|_{\tilde{L}_t^{\infty}(\dot{B}_{p,1}^{-1+3/p})} 
\|a\|_{\tilde{L}_t^2(\dot{B}_{p,1}^{3/p})} 
(1+
\|a\|_{\tilde{L}_t^{\infty}(\dot{B}_{p,1}^{3/p})})
\lesssim \mathcal{X} \mathcal{Y}(1+ \mathcal{X})\leq C\mathcal{X} \mathcal{Y}.
\end{aligned}
\end{equation}
By Taylor's formula, 
\[
I(a) =I'(0)a + \tilde{h}(a)a,\,  \tilde{h}(0) = 0
\]
and given $2\leq p<6,\,\, 2\leq q\leq p \leq 2q, \,\,\frac{3}{q} - \frac{3}{p} \leq 1$,
by Lemma \ref{lem key} and Lemma \ref{linestimate}, we have
\begin{equation}\label{ha}
\begin{aligned}
&\|(\tilde{h}(a)a)^{\ell}\|_{\tilde{L}_{t}^{2}(\dot{B}_{q,1}^{\frac{5}{q}-1})}
\lesssim\|(T_{a}\tilde{h}(a))^{\ell}+(T_{g(a)}a)^{\ell}+(R(a,\tilde{h}(a)))^{\ell}\|_{\tilde{L}_{t}^{2}(\dot{B}_{q,1}^{\frac{3}{q}-1})}
\\
&\lesssim\|a\|_{\tilde{L}_{t}^{\infty}(\dot{B}_{p,1}^{-1+\frac{3}{p}})}
\|\tilde{h}(a)\|_{\tilde{L}_{t}^{2}(\dot{B}_{p,1}^{\frac{3}{p}})}
+\|\tilde{h}(a)\|_{\tilde{L}_{t}^{\infty}(\dot{B}_{p,1}^{-1+\frac{3}{p}})}
\|a\|_{\tilde{L}_{t}^{2}(\dot{B}_{p,1}^{\frac{3}{p}})}
\\
&
\lesssim(1+\|a\|_{\tilde{L}_{t}^{\infty}(\dot{B}_{p,1}^{\frac{3}{p}})})\|a\|_{\tilde{L}_{t}^{\infty}(\dot{B}_{p,1}^{\frac{3}{p}-1})}
\|a\|_{\tilde{L}_{t}^{2}(\dot{B}_{p,1}^{\frac{3}{p}})}
\\
&\lesssim (1+\mathcal{X})\mathcal{X}^{2}
\end{aligned}
\end{equation}
from which we get
\begin{equation}\label{ia} 
\begin{aligned}
&\|(I(a))^\ell\|_{\tilde{L}_{t}^{2}(\dot{B}_{q,1}^{\frac{5}{q}-1})}
\lesssim\|(I'(0)a)^{\ell}\|_{\tilde{L}_{t}^{2}(\dot{B}_{q,1}^{\frac{5}{q}-1})}+
\|(\tilde{h}(a)a)^{\ell}\|_{\tilde{L}_{t}^{2}(\dot{B}_{q,1}^{\frac{5}{q}-1})}
\\
 &
 \lesssim \mathcal{X}(1+ \mathcal{X}+\mathcal{X}^{2})\leq C\mathcal{X}.
\end{aligned}
\end{equation}
Combining \eqref{m1} and \eqref{ia} yields
\begin{equation}\label{ma}
\begin{aligned}
&\|(T_m (I(a))^{\ell})^\ell\|_{\tilde{L}_t^1(\dot{B}_{q,1}^{5/q})} 
\lesssim\|m\|_{\tilde{L}_t^{2}(L^{\infty})} \|(I(a))^{\ell}
\|_{\tilde{L}_t^2(\dot{B}_{q,1}^{\frac{5}{q}})}\\
 &\lesssim\|m\|_{\tilde{L}_t^{2}(\dot{B}_{p,1}^{\frac{3}{p}})} \|(I(a))^{\ell}
\|_{\tilde{L}_t^2(\dot{B}_{q,1}^{-1+\frac{5}{q}})}\leq C\mathcal{X}^{2}.
\end{aligned}
\end{equation}
Use the decomposition $$(T_{m} (I(a))^h)^{\ell}=
 [S_{k_0}, T_{m}] (I(a))^h + T_{m} S_{k_0} (I(a))^h.$$ 
We now focus on the commutator term. For $2 \leq q \leq p \leq 2q$ , the term  \eqref{m1} combined with Lemma~\ref{commo} yields
\begin{equation} \label{ma1}
\begin{aligned}
&\|([S_{k_0}, T_{m}] (I(a))^h)^{\ell}
\|_{\tilde{L}_{t}^{1}(\dot{B}_{q,1}^{\frac{5}{q}})}\lesssim
\|([S_{k_0}, T_{m}] (I(a))^h)^{\ell}
\|_{\tilde{L}_{t}^{1}(\dot{B}_{q,1}^{\frac{3}{q}})}
\\
&\lesssim\|\nabla m\|_{\tilde{L}_{t}^{2}(\dot{B}_{p,1}^{\frac{3}{p}-1})}
\|(I(a))^h\|_{\tilde{L}_{t}^{2}(\dot{B}_{p,1}^{\frac{3}{p}})}\\
&\lesssim\| m\|_{\tilde{L}_{t}^{2}(\dot{B}_{p,1}^{\frac{3}{p}})}
\|a\|_{\tilde{L}_{t}^{2}(\dot{B}_{p,1}^{\frac{3}{p}})}(1+
\|a\|_{\tilde{L}_t^{\infty}(\dot{B}_{p,1}^{3/p})})\leq C\mathcal{X}^{2}.
\end{aligned}
\end{equation}
Similar to \eqref{ma}, we also have
\begin{equation} \label{ma2}
\begin{aligned}
\|T_{m} S_{k_0} (I(a))^h)
\|_{\tilde{L}_{t}^{1}(\dot{B}_{q,1}^{\frac{5}{q}})}\leq C\mathcal{X}^{2}.
\end{aligned}
\end{equation}
Gathering the results from  \eqref{ma}, \eqref{ma1}, and \eqref{ma2}, we derive
\begin{align}\label{s1}
\|(T_m I(a))^\ell\|_{\tilde{L}_t^1(\dot{B}_{q,1}^{5/q})} 
\leq C\mathcal{X}^{2}.
\end{align}
Since $\dot{B}_{p,1}^{3/p} \hookrightarrow \dot{B}_{\infty,1}^0$, applying Remark~\ref{remarkbon} to the paraproduct terms yields
\begin{equation}\label{T1}
\begin{aligned}
\|(T_{I(a)} m^\ell)^\ell\|_{\tilde{L}_t^2(\dot{B}_{q,1}^{-1+5/q})} 
&\lesssim\|I(a)\|_{\tilde{L}_t^{\infty}(\dot{B}_{\infty,1}^0)} 
\|m^\ell\|_{\tilde{L}_t^2(\dot{B}_{q,1}^{-1+5/q})} \leq C\mathcal{X}\mathcal{Y}, \\
\|(T_{I(a)} m^\ell)^\ell\|_{\tilde{L}_t^1(\dot{B}_{q,1}^{5/q})} 
&\lesssim\|I(a)\|_{\tilde{L}_t^{\infty}(\dot{B}_{\infty,1}^0)}
 \|m^\ell\|_{\tilde{L}_t^1(\dot{B}_{q,1}^{5/q})}\leq C\mathcal{X}\mathcal{Y},
\end{aligned}
\end{equation}
and by \eqref{m1}, Lemma \ref{lem key} and 
Lemma \ref{linestimate}, for $2\leq q\leq p \leq 2q$, $\frac{3}{q} - \frac{3}{p} \leq 1$, we have
\begin{equation}\label{T2}
\begin{aligned}
\|(T_{I(a)} m^h)^\ell\|_{\tilde{L}_t^2(\dot{B}_{q,1}^{-1+5/q})} 
&\lesssim \|(T_{I(a)} m^h)^\ell\|_{\tilde{L}_t^2(\dot{B}_{q,1}^{-1+3/q})}\\
&\leq \|I(a)\|_{\tilde{L}_t^{\infty}(\dot{B}_{p,1}^{-1+3/p})} 
\|m^h\|_{\tilde{L}_t^2(\dot{B}_{p,1}^{3/p})} \leq C\mathcal{X}^{2}, \\
\|(T_{I(a)} m^h)^\ell\|_{\tilde{L}_t^1(\dot{B}_{q,1}^{5/q})} 
&\lesssim \|(T_{I(a)} m^h)^\ell\|_{\tilde{L}_t^1(\dot{B}_{q,1}^{-1+3/q})} \\
&\leq \|I(a)\|_{\tilde{L}_t^{\infty}(\dot{B}_{p,1}^{-1+3/p})} 
\|m^h\|_{\tilde{L}_t^1(\dot{B}_{p,1}^{3/p})}\leq C\mathcal{X}\mathcal{Y}.
\end{aligned}
\end{equation}
Combining \eqref{s0}, \eqref{s1}, \eqref{T1}, and \eqref{T2} yields
\begin{equation}\label{T3}
\begin{aligned}
\|(T_{I(a)} m + T_m I(a))^\ell
\|_{\tilde{L}_t^2(\dot{B}_{q,1}^{-1+5/q})
 \cap \tilde{L}_t^1(\dot{B}_{q,1}^{5/q})}
 \leq C\mathcal{X}\mathcal{Y}+C\mathcal{X}^{2}.
\end{aligned}
\end{equation}
Regarding the remainder term for parameters satisfying $2 \leq p < 6$ and $ q\leq p \leq 2q$, applying \eqref{m1} along with Lemma \ref{lem key} yields
\begin{equation}\label{T4}
\begin{aligned}
&\|(R(I(a), m))^\ell\|_{\tilde{L}_t^2(\dot{B}_{q,1}^{-1+5/q})} 
\lesssim\|(R(I(a), m))^\ell\|_{\tilde{L}_t^2(\dot{B}_{q,1}^{-1+3/q})} \\
& \  \  \  \ \lesssim\|I(a)\|_{\tilde{L}_t^2(\dot{B}_{p,1}^{3/p})} 
\|m\|_{\tilde{L}_t^{\infty}(\dot{B}_{p,1}^{-1+3/p})} \leq C\mathcal{X}\mathcal{Y}, \\
&\|(R(I(a), m))^\ell\|_{\tilde{L}_t^1(\dot{B}_{q,1}^{5/q})} 
\lesssim\|(R(I(a), m))^\ell\|_{\tilde{L}_t^1(\dot{B}_{q,1}^{3/q})} \\
& \  \  \  \ \lesssim\|I(a)\|_{\tilde{L}_t^2(\dot{B}_{p,1}^{3/p})} 
\|m\|_{\tilde{L}_t^2(\dot{B}_{p,1}^{3/p})}
 \leq C\mathcal{X}^{2}.
\end{aligned}
\end{equation}
Using \eqref{qu} and combining \eqref{T3} with \eqref{T4}, we obtain
\begin{align}\label{T5}
\|(I(a)m)^\ell\|_{\tilde{L}_t^2(\dot{B}_{q,1}^{-1+5/q}) \cap 
\tilde{L}_t^1(\dot{B}_{q,1}^{5/q})} \leq C\mathcal{X}^{2},
\end{align}
from which and the relation $u = m - I(a)m$, we finally derive
\begin{align*}
&\|u^\ell\|_{\tilde{L}_t^{\infty}(\dot{B}_{q,1}^{-1+3/q}) 
\cap \tilde{L}_t^2(\dot{B}_{q,1}^{-1+5/q}) \cap \tilde{L}_t^1(\dot{B}_{q,1}^{5/q})}\leq \|m^\ell\|_{\tilde{L}_t^{\infty}(\dot{B}_{q,1}^{-1+3/q}) 
\cap \tilde{L}_t^2(\dot{B}_{q,1}^{-1+5/q}) \cap \tilde{L}_t^1(\dot{B}_{q,1}^{5/q})} \nonumber \\
&\quad + \|(I(a)m)^\ell\|_{\tilde{L}_t^{\infty}(\dot{B}_{q,1}^{-1+3/q}) 
\cap \tilde{L}_t^1(\dot{B}_{q,1}^{5/q})} 
\leq\mathcal{Y}^{\ell}+C\mathcal{X}^{2}.
\end{align*}
This concludes the proof.
\end{proof}

\subsubsection{\bf Low frequency estimate for linear part}\label{linear part}
In this subsection, we present the linear estimates for low frequencies.
In light of Proposition \ref{momentum}, now it is enough to control $\mathcal{Y}^\ell$. 
For the coupled system, applying Lemma \ref{low} to the momentum equations \eqref{comnscauchy}
 yield the following estimates for $q \geq 2$:
\begin{equation}\label{amlow}
\begin{aligned}
&\|(a, m)^{\ell}\|_{\tilde{L}_{t}^{\infty}(\dot{B}_{q,1}^{-1 + \frac{3}{q}})} 
 + \|(a, m)^{\ell}\|_{\tilde{L}_{t}^{1}(\dot{B}_{q,1}^{5/q})} 
 + \|(a, m)^{\ell}\|_{\tilde{L}_{t}^{2}(\dot{B}_{q,1}^{-1+ \frac{5}{q}})} \\
&\lesssim \|(a_{0}, m_{0})^{\ell}\|_{\dot{B}_{q,1}^{-2 + \frac{5}{q}}} 
 + \|(a_{0},  m_{0})^{\ell}\|_{\dot{B}_{q,1}^{-3 + \frac{7}{q}}}  \\
&\quad + \| (h(a, m))^{\ell}\|_{\tilde{L}_{t}^{1}(\dot{B}_{q,1}^{\frac{5}{q} - 2})} 
 + \| (h(a, m))^{\ell}\|_{\tilde{L}_{t}^{1}(\dot{B}_{q,1}^{\frac{7}{q} - 3})} 
  \\
&\lesssim \|(a_{0},  m_{0})^{\ell}\|_{\dot{B}_{q,1}^{-3 + \frac{7}{q}}} 
 + \|(h(a, m))^{\ell}\|_{\tilde{L}_{t}^{1}(\dot{B}_{q,1}^{\frac{7}{q}- 3})}.
 \end{aligned}
\end{equation}
Using Proposition~\ref{momentum}, we have
\begin{equation*}
\|u^{\ell}\|_{\tilde{L}_{t}^{\infty}(\dot{B}_{q,1}^{-1
+\frac{3}{q}})\cap\tilde{L}_{t}^{2}(\dot{B}_{q,1}^{-1+\frac{5}{q}})
\cap\tilde{L}_{t}^{1}(\dot{B}_{q,1}^{\frac{5}{q}})}
\lesssim\mathcal{Y}^{\ell}+C\mathcal{X}^{2},
\end{equation*}
which directly implies
\begin{equation}\label{aulow}
\begin{aligned}
\|(a,u)^{\ell}\|_{\tilde{L}_{t}^{\infty}(\dot{B}_{q,1}^{-1 + \frac{3}{q}}) 
    \cap \tilde{L}_{t}^{2}(\dot{B}_{q,1}^{-1 + \frac{5}{q}}) 
    \cap \tilde{L}_{t}^{1}(\dot{B}_{q,1}^{\frac{5}{q}})}
&\lesssim \|(a,m)^{\ell}\|_{\tilde{L}_{t}^{\infty}(\dot{B}_{q,1}^{-1 + \frac{3}{q}}) 
    \cap \tilde{L}_{t}^{2}(\dot{B}_{q,1}^{-1 + \frac{5}{q}}) 
    \cap \tilde{L}_{t}^{1}(\dot{B}_{q,1}^{\frac{5}{q}})}+\mathcal{X}^{2}.
\end{aligned}
\end{equation}
Combining 
\eqref{amlow}, \eqref{aulow} and Proposition~\ref{momentum1}, then we have the following low estimates
\begin{equation*}
\begin{aligned}
&\|(a,u)^{\ell}\|_{\tilde{L}_{t}^{\infty}(\dot{B}_{q,1}^{-1 + \frac{3}{q}}) 
    \cap \tilde{L}_{t}^{2}(\dot{B}_{q,1}^{-1 + \frac{5}{q}}) 
    \cap \tilde{L}_{t}^{1}(\dot{B}_{q,1}^{\frac{5}{q}})} \\
&\lesssim 
\|(a_{0}, m_{0})^{\ell}\|_{\dot{B}_{q,1}^{-3 + \frac{7}{q}}}+ 
\|(h(a, m))^{\ell}\|_{\tilde{L}_{t}^{1}(\dot{B}_{q,1}^{-3 + \frac{7}{q}})} 
    + \mathcal{X}^{2}.
\end{aligned}
\end{equation*}

\subsubsection{\bf Low frequency estimate for nonlinear terms}\label{nonlinear part}

\begin{prop}\label{momentum1}
Assume $2 \leq p< 6,\,\,2 \leq q \leq p \leq 2q$, and that $(q,\,\, p)$ satisfies
\begin{align*}
\frac{3}{q}-\frac{3}{p}\leq 1<\frac{2}{q}+\frac{3}{p},
\end{align*} 
Suppose that \((a, u)\) is a solution of system \eqref{CNS} with \(\mathcal{X} \lesssim 1\) and \(\|a\|_{\tilde{L}_{t}^{\infty}(\dot{B}_{p,1}^{3/p})} \lesssim 1\). Then we have
\begin{equation*}
\label{g}
\|(h(a, m))^{\ell}\|_{\tilde{L}_{t}^{1}(\dot{B}_{q,1}^{\frac{7}{q} - 3})} \lesssim \mathcal{X}^{2}.
\end{equation*}
\end{prop}

\begin{proof}

\textbf{Step 1: To estimate $\|(\mathrm{div}(m\otimes m))^{\ell}\|_{\tilde{L}_{t}^{1}(\dot{B}_{q,1}^{\frac{7}{q}-3})}$.}
we analyze the terms of quadratic convection. First observe that
\begin{equation}\label{mm1}
\|(\mathrm{div}(m\otimes m))^{\ell}\|_{\tilde{L}_{t}^{1}(\dot{B}_{q,1}^{\frac{7}{q}-3})}
\lesssim\|(m\otimes m)^{\ell}\|_{\tilde{L}_{t}^{1}(\dot{B}_{q,1}^{\frac{7}{q}-2})}.
\end{equation}
Using Bony's decomposition, for $q\geq 2$, 
Remark~\ref{remarkbon} and the term \eqref{m1} yield
\begin{equation}\label{mm11} 
\begin{aligned}
\|(T_{m} m^\ell)^{\ell}\|_{\tilde{L}_{t}^{1}(\dot{B}_{q,1}^{\frac{7}{q}-2})}
&\lesssim\|m\|_{\tilde{L}_{t}^{2}(\dot{B}_{\infty,1}^{\frac{2}{q}-1})}
\|m^\ell\|_{\tilde{L}_{t}^{2}(\dot{B}_{q,1}^{\frac{5}{q}-1})} \\
&\lesssim\|m^\ell\|_{\tilde{L}_{t}^{2}(\dot{B}_{q,1}^{\frac{5}{q}-1})}^2 
+ \|m^h\|_{\tilde{L}_{t}^{2}(\dot{B}_{\infty,1}^{0})}
\|m^\ell\|_{\tilde{L}_{t}^{2}(\dot{B}_{q,1}^{\frac{5}{q}-1})} \\
&\lesssim\|m^\ell\|_{\tilde{L}_{t}^{2}(\dot{B}_{q,1}^{\frac{5}{q}-1})}^2 
+ \|m^h\|_{\tilde{L}_{t}^{2}(\dot{B}_{p,1}^{\frac{3}{p}})}
\|m^\ell\|_{\tilde{L}_{t}^{2}(\dot{B}_{q,1}^{\frac{5}{q}-1})}\lesssim 
\mathcal{Y}^{2}+\mathcal{X}\mathcal{Y}.
\end{aligned}
\end{equation}
Next, we decompose
\[
(T_{m} m^h)^\ell = [S_{k_{0}}, T_{m}]m^h + T_{m}S_{k_{0}}m^h.
\]
 Similar to (\ref{m1}), we have
\begin{equation}\label{mm2} 
\begin{aligned}
\|T_{m}S_{k_{0}}m^h\|_{\tilde{L}_{t}^{1}(\dot{B}_{q,1}^{\frac{7}{q}-2})}
\lesssim 
\mathcal{Y}^{2}+\mathcal{X}\mathcal{Y}.
\end{aligned}
\end{equation}
For $2\leq q\leq p\leq 2q$,  Lemma~\ref{commo} and (\ref{m1}) imply 
\begin{equation}\label{mm3}
\begin{aligned}
\|[S_{k_{0}}, T_{m}]m^h\|_{\tilde{L}_{t}^{1}(\dot{B}_{q,1}^{\frac{7}{q}-2})}
\lesssim\|\nabla m\|_{\tilde{L}_{t}^{2}(\dot{B}_{p,1}^{\frac{3}{p}-1})}
\|m^h\|_{\tilde{L}_{t}^{2}(\dot{B}_{p,1}^{\frac{3}{p} + \frac{4}{q} - 2})}\lesssim\|m\|^2_{\tilde{L}_{t}^{2}(\dot{B}_{p,1}^{\frac{3}{p}})}\lesssim 
\mathcal{X}^{2}.
\end{aligned}
\end{equation}
For the remainder term, when $\frac{2}{q} + \frac{3}{p} > 1$ and  $2\leq q\leq p\leq 2q$, Lemma \ref{lem key} gives
\begin{equation}\label{mm4} 
\begin{aligned}
\|(R(m^h, m^h))^{\ell}\|_{\tilde{L}_{t}^{1}(\dot{B}_{q,1}^{\frac{7}{q}-2})}\lesssim\|m^h\|_{\tilde{L}_{t}^{2}(\dot{B}_{p,1}^{\frac{3}{p}})}
\|m^h\|_{\tilde{L}_{t}^{2}(\dot{B}_{p,1}^{\frac{3}{p} + \frac{4}{q} - 2})}\lesssim\|m\|^2_{\tilde{L}_{t}^{2}(\dot{B}_{p,1}^{\frac{3}{p}})} \lesssim 
\mathcal{X}^{2}.
\end{aligned}
\end{equation}
Under the conditions $p \geq q$, the inequality $\frac{3}{p} + \frac{2}{q} > 1$ implies $\frac{10}{q} - 2 > 0$. Applying Lemma \ref{lem key} to this result, we obtain
\begin{equation}\label{mm5} 
\|(R(m^\ell, m^\ell))^{\ell}\|_{\tilde{L}_{t}^{1}(\dot{B}_{q,1}^{\frac{7}{q}-2})}
\lesssim\|m^\ell\|^2_{\tilde{L}_{t}^{2}(\dot{B}_{q,1}^{\frac{5}{q}-1})}
\lesssim 
\mathcal{Y}^{2}.
\end{equation}
Combining these estimates (\ref{mm1})-(\ref{mm5}) and Proposition ~\ref{momentum}, we have
\begin{equation*}
\begin{aligned}
&\|(\mathrm{div}(m\otimes m))^{\ell}\|_{\tilde{L}_{t}^{1}(\dot{B}_{q,1}^{\frac{7}{q}-3})} \\
&\lesssim\|(T_{m} m^\ell+T_{m} m^h+R(m^h, m^h)+R(m^\ell, m^\ell))^{\ell}
\|_{\tilde{L}_{t}^{1}(\dot{B}_{q,1}^{\frac{7}{q}-3})}
 \\
&\lesssim 
\mathcal{Y}^{2}+\mathcal{X}\mathcal{Y}+\mathcal{X}^{2}\lesssim \mathcal{X}^{2}.
\end{aligned}
\end{equation*}

\textbf{Step 2: To estimate $
\|(\mathrm{div}(I(a)m\otimes m))^{\ell}\|_{\tilde{L}_{t}^{1}(\dot{B}_{q,1}^{\frac{7}{q}-3})}$.}

\textbf{Case 1: To estimate $\|(T_{I(a)m} m^\ell)^{\ell}
\|_{\tilde{L}_{t}^{1}(\dot{B}_{q,1}^{\frac{7}{q}-2})}$.} 
From (\ref{m1}), Lemma \ref{prodlemma} and Lemma~\ref{linestimate}, it holds that
\begin{equation}\label{iam1}
\begin{aligned}
&\|(I(a)m)^{h}\|_{\tilde{L}_{t}^{2}(\dot{B}_{p,1}^{\frac{3}{p}})}\lesssim \|I(a)\|_{\tilde{L}_{t}^{\infty}(\dot{B}_{p,1}^{
\frac{3}{p}})}\|m\|_{\tilde{L}_{t}^{2}(\dot{B}_{p,1}^{
\frac{3}{p}})}\lesssim\mathcal{X}^{2}.
\end{aligned}
\end{equation}
For $q\geq2$, by  (\ref{m1}) and Lemma~\ref{linestimate} and Remark~\ref{remarkbon}, we have
\begin{equation}\label{iam2}
\begin{aligned}
&\|(T_{I(a)}m)^{\ell}\|_{\tilde{L}_{t}^{2}(\dot{B}_{\infty,1}^{-1
+\frac{2}{q}})}\lesssim \|(T_{I(a)}m)^{\ell}\|_{\tilde{L}_{t}^{2}(\dot{B}_{p,1}^{-1
+\frac{2}{q}+\frac{3}{p}})}
\\
&\lesssim \|I(a)\|_{\tilde{L}_{t}^{\infty}(L^{\infty})}(\|m^{\ell}\|_{\tilde{L}_{t}^{2}(\dot{B}_{p,1}^{-1
+\frac{2}{q}+\frac{3}{p}})}+\|m^{h}\|_{\tilde{L}_{t}^{2}(\dot{B}_{p,1}^{-1
+\frac{2}{q}+\frac{3}{p}})})\\
&\lesssim \|I(a)\|_{\tilde{L}_{t}^{\infty}(L^{\infty})}(\|m^{\ell}\|_{\tilde{L}_{t}^{2}(\dot{B}_{q,1}^{-1
+\frac{5}{q}})}+\|m^{h}\|_{\tilde{L}_{t}^{2}(\dot{B}_{p,1}^{-1
+\frac{2}{q}+\frac{3}{p}})})\\
&\lesssim \|a\|_{\tilde{L}_{t}^{\infty}(\dot{B}_{p,1}^{\frac{3}{p}})}(\|m^{\ell}\|_{\tilde{L}_{t}^{2}(\dot{B}_{q,1}^{-1
+\frac{5}{q}})}+
\|m^{h}\|_{\tilde{L}_{t}^{2}(\dot{B}_{p,1}^{\frac{3}{p}})})\lesssim \mathcal{X}^{2}
+\mathcal{X}\mathcal{Y}.
\end{aligned}
\end{equation}
By Lemma \ref{lem key} and Lemma~\ref{linestimate}, for $2\leq p<6,\, 2\leq q\leq p\leq 2q$, we have
\begin{equation}\label{iam3}
\begin{aligned}
&\|(R(I(a),m))^{\ell}\|_{\tilde{L}_{t}^{2}(\dot{B}_{\infty,1}^{-1
+\frac{2}{q}})}\\
&\lesssim \|(R(I(a),m))^{\ell}\|_{\tilde{L}_{t}^{2}(\dot{B}_{q,1}^{-1
+\frac{5}{q}})}\lesssim \|(R(I(a),m)^{\ell}\|_{\tilde{L}_{t}^{2}(\dot{B}_{q,1}^{-1
+\frac{3}{q}})}
\\
&\lesssim \|I(a)\|_{\tilde{L}_{t}^{2}(\dot{B}_{p,1}^{
\frac{3}{p}})}(\|m^{\ell}\|_{\tilde{L}_{t}^{\infty}(\dot{B}_{p,1}^{
-1+\frac{3}{p}})}+\|m^{h}\|_{\tilde{L}_{t}^{\infty}(\dot{B}_{p,1}^{
-1+\frac{3}{p}})})\\
&\lesssim \|I(a)\|_{\tilde{L}_{t}^{2}(\dot{B}_{p,1}^{
\frac{3}{p}})}(\|m^{\ell}\|_{\tilde{L}_{t}^{\infty}(\dot{B}_{q,1}^{
-1+\frac{3}{q}})}+\|m^{h}\|_{\tilde{L}_{t}^{\infty}(\dot{B}_{p,1}^{
-1+\frac{3}{p}})})\lesssim \mathcal{X}\mathcal{Y}.\\
\end{aligned}
\end{equation}
For  $ \frac{3}{q}-\frac{3}{p}\leq1,\,2\leq q\leq p\leq 2q$, by 
Lemma \ref{lem key} and Lemma~\ref{linestimate}, we have
\begin{equation}\label{iam4}
\begin{aligned}
&\|(T_{m}I(a))^{\ell}\|_{\tilde{L}_{t}^{2}(\dot{B}_{\infty,1}^{-1
+\frac{2}{q}})}\lesssim \|(T_{m}I(a))^{\ell}\|_{\tilde{L}_{t}^{2}(\dot{B}_{q,1}^{-1
+\frac{5}{q}})}\lesssim \|(T_{m}I(a))^{\ell}\|_{\tilde{L}_{t}^{2}(\dot{B}_{q,1}^{-1
+\frac{3}{q}})}
\\
&\lesssim \|m\|_{\tilde{L}_{t}^{\infty}(\dot{B}_{p,1}^{-1
+\frac{3}{p}})}\|I(a)\|_{\tilde{L}_{t}^{2}(\dot{B}_{p,1}^{\frac{3}{p}})}\lesssim 
\mathcal{X}\mathcal{Y}.
\end{aligned}
\end{equation}
By  (\ref{iam1})-(\ref{iam4}), we thus obtain for $q\geq2$
\begin{equation}\label{iam5}
\begin{aligned}
&\|I(a)m\|_{\tilde{L}_{t}^{2}(\dot{B}_{\infty,1}^{-1
+\frac{2}{q}})}\lesssim \|(R(I(a),m)+T_{m}I(a)+T_{I(a)}m)^{\ell}\|_{\tilde{L}_{t}^{2}(\dot{B}_{\infty,1}^{-1
+\frac{2}{q}})}\\
& \ \ +\|(I(a)m)^{h}\|_{\tilde{L}_{t}^{2}(\dot{B}_{p,1}^{-1
+\frac{2}{q}+\frac{3}{p}})}\lesssim \mathcal{X}^{2}
+\mathcal{X}\mathcal{Y}+\|(I(a)m)^{h}\|_{\tilde{L}_{t}^{2}(\dot{B}_{p,1}^{\frac{3}{p}})}
\\
&\lesssim\mathcal{X}^{2}+\mathcal{X}\mathcal{Y}.
\end{aligned}
\end{equation}
Remark~\ref{remarkbon}, Proposition ~\ref{momentum} and (\ref{iam5}) yield
 \begin{equation}\label{iam6}
\begin{aligned}
&\|(T_{I(a)m} m^\ell)^{\ell}\|_{\tilde{L}_{t}^{1}(\dot{B}_{q,1}^{\frac{7}{q}-2})}
\lesssim
\|I(a)m\|_{\tilde{L}_{t}^{2}(\dot{B}_{\infty,1}^{\frac{2}{q}-1})}
\|m^\ell\|_{\tilde{L}_{t}^{2}(\dot{B}_{q,1}^{\frac{5}{q}-1})}
\\
&\lesssim\mathcal{X}\mathcal{Y}(\mathcal{X}+\mathcal{Y})\lesssim \mathcal{X}^{3}.
\end{aligned}
\end{equation}

\textbf{Case 2: To estimate $\|(T_{I(a)m} m^h)^{\ell}
\|_{\tilde{L}_{t}^{1}(\dot{B}_{q,1}^{\frac{7}{q}-2})}$.}
Next, we decompose
\[
(T_{I(a)m} m^h)^\ell = [S_{k_{0}}, T_{I(a)m}]m^h + T_{I(a)m}S_{k_{0}}m^h.
\]
 Similar to (\ref{iam6}), we have
\begin{equation*}
\begin{aligned}
\|T_{I(a)m}S_{k_{0}}m^h\|_{\tilde{L}_{t}^{1}(\dot{B}_{q,1}^{\frac{7}{q}-2})}
\lesssim \mathcal{X}^{3}.
\end{aligned}
\end{equation*}
For $2\leq q\leq p\leq 2q$, with the help of Lemma~\ref{commo}, Lemma~\ref{linestimate} and (\ref{m1}),
\begin{equation*}
\begin{aligned}
\|[S_{k_{0}}, T_{I(a)m}]m^h\|_{\tilde{L}_{t}^{1}(\dot{B}_{q,1}^{\frac{7}{q}-2})}
&\lesssim\|\nabla (I(a)m)\|_{\tilde{L}_{t}^{2}(\dot{B}_{p,1}^{\frac{3}{p}-1})}
\|m^h\|_{\tilde{L}_{t}^{2}(\dot{B}_{p,1}^{\frac{3}{p} + \frac{4}{q} - 2})} \\
&\lesssim\|I(a)\|_{\tilde{L}_{t}^{\infty}(\dot{B}_{p,1}^{
\frac{3}{p}})}\|m\|^2_{\tilde{L}_{t}^{2}(\dot{B}_{p,1}^{\frac{3}{p}})} 
\lesssim \mathcal{X}^{3}.
\end{aligned}
\end{equation*}
From above, we thus obtain
\begin{eqnarray}\label{iam10}
\|(T_{I(a)m} m^h)^{\ell}\|_{\tilde{L}_{t}^{1}(\dot{B}_{q,1}^{\frac{7}{q}-2})}
\lesssim \mathcal{X}^{3}.
\end{eqnarray}

\textbf{Case 3: To estimate $\|(T_{m} (I(a)m)^\ell)^{\ell}
\|_{\tilde{L}_{t}^{1}(\dot{B}_{q,1}^{\frac{7}{q}-2})}$.} 
By (\ref{m1}), for $q\geq2$, we have
 \begin{equation}\label{iam7}
\begin{aligned}
&\|m\|_{\tilde{L}_{t}^{2}(\dot{B}_{\infty,1}^{-1
+\frac{2}{q}})}\lesssim \|m^{\ell}\|_{\tilde{L}_{t}^{2}(\dot{B}_{\infty,1}^{-1
+\frac{2}{q}})}+\|m^{h}\|_{\tilde{L}_{t}^{2}(\dot{B}_{p,1}^{\frac{3}{p}})}\\
&\lesssim \|m\|_{\tilde{L}_{t}^{2}(\dot{B}_{q,1}^{-1
+\frac{5}{q}})}^{\ell}+
\|m\|_{\tilde{L}_{t}^{2}(\dot{B}_{p,1}^{\frac{3}{p}})}^{h}\lesssim \mathcal{X}+\mathcal{Y}.
\end{aligned}
\end{equation}
For $q\geq2$, from (\ref{T5}),  (\ref{iam7}), Remark~\ref{remarkbon} and Proposition ~\ref{momentum}, it thus holds that
 \begin{equation}\label{iam8}
\begin{aligned}
&\|T_{m} (I(a)m)^\ell\|_{\tilde{L}_{t}^{1}(\dot{B}_{q,1}^{\frac{7}{q}-2})}^{\ell}
\lesssim\|m\|_{\tilde{L}_{t}^{2}(\dot{B}_{\infty,1}^{\frac{2}{q}-1})}
\|(I(a)m)^\ell\|_{\tilde{L}_{t}^{2}(\dot{B}_{q,1}^{\frac{5}{q}-1})}
\\
&\lesssim \mathcal{X}^{2}(\mathcal{X}+\mathcal{Y})\lesssim \mathcal{X}^{3}.
\end{aligned}
\end{equation}

\textbf{Case 4: To estimate $\|(T_{m} (I(a)m)^h)^{\ell}
\|_{\tilde{L}_{t}^{1}(\dot{B}_{q,1}^{\frac{7}{q}-2})}$.}
Next, we decompose
\[
(T_{m} (I(a)m)^h)^\ell = [S_{k_{0}}, T_{m}](I(a)m)^h + T_{m}S_{k_{0}}(I(a)m)^h.
\]
 Similar to (\ref{iam8}), we obtain
\begin{equation*}
\begin{aligned}
\|T_{m}S_{k_{0}}(I(a)m)^h\|_{\tilde{L}_{t}^{1}(\dot{B}_{q,1}^{\frac{7}{q}-2})}
\lesssim\mathcal{X}^{3}.
\end{aligned}
\end{equation*}
For $q\geq2$, by Lemma~\ref{commo}, Lemma~\ref{linestimate} and (\ref{m1}),
we have
\begin{equation*}
\begin{aligned}
\|[S_{k_{0}}, T_{m}](I(a)m)^h\|_{\tilde{L}_{t}^{1}(\dot{B}_{q,1}^{\frac{7}{q}-2})}
&\lesssim\|\nabla m\|_{\tilde{L}_{t}^{2}(\dot{B}_{p,1}^{\frac{3}{p}-1})}
\|(I(a)m)^h\|_{\tilde{L}_{t}^{2}(\dot{B}_{p,1}^{\frac{3}{p} + \frac{4}{q}-2})} \\
&\lesssim\|I(a)\|_{\tilde{L}_{t}^{\infty}(\dot{B}_{p,1}^{
\frac{3}{p}})}\|m\|^2_{\tilde{L}_{t}^{2}(\dot{B}_{p,1}^{\frac{3}{p}})} 
\lesssim\mathcal{X}^{3}.
\end{aligned}
\end{equation*}
From above, we have the following
\begin{eqnarray}\label{iam9}
\|(T_{m} (I(a)m)^h)^{\ell}\|_{\tilde{L}_{t}^{1}(\dot{B}_{q,1}^{\frac{7}{q}-2})}
\lesssim \mathcal{X}^{3}.
\end{eqnarray}

\textbf{Case 5: To estimate $
\|(R(m,I(a)m))^{\ell}\|_{\tilde{L}_{t}^{1}(\dot{B}_{q,1}^{\frac{7}{q}-2})}$.}
  When $\frac{2}{q}+\frac{3}{p}>1$ and $q\geq 2$, Lemma \ref{prodlemma}, Lemma~\ref{linestimate}, Lemma \ref{lem key} and (\ref{m1}) give 
\begin{equation}\label{ram1}
\begin{aligned}
\|(R(m^h, (I(a)m)^h))^{\ell}\|_{\tilde{L}_{t}^{1}(\dot{B}_{q,1}^{\frac{7}{q}-2})}
&\lesssim \|m^h\|_{\tilde{L}_{t}^{2}(\dot{B}_{p,1}^{\frac{3}{p}})}
\|(I(a)m)^h\|_{\tilde{L}_{t}^{2}(\dot{B}_{p,1}^{\frac{3}{p} + \frac{4}{q} - 2})} \\
&\lesssim\|a\|_{\tilde{L}_{t}^{\infty}(\dot{B}_{p,1}^{\frac{3}{p}})}
\|m\|^2_{\tilde{L}_{t}^{2}(\dot{B}_{p,1}^{\frac{3}{p}})} \lesssim\mathcal{X}^{3}.
\end{aligned}
\end{equation}
When $q<5$, by Lemma \ref{lem key} and (\ref{T5}),
\begin{equation}\label{ram2}
\begin{aligned}
\|(R(m^\ell, (I(a)m)^\ell))^{\ell}\|_{\tilde{L}_{t}^{1}(\dot{B}_{q,1}^{\frac{7}{q}-2})}
\lesssim\|m^\ell\|_{\tilde{L}_{t}^{2}(\dot{B}_{q,1}^{\frac{5}{q}-1})}
\| (I(a)m)^\ell\|_{\tilde{L}_{t}^{2}(\dot{B}_{q,1}^{\frac{5}{q}-1})}
\lesssim\mathcal{X}^{2}\mathcal{Y}.
\end{aligned}
\end{equation}
 By  (\ref{ram1}), (\ref{ram2}), and Proposition ~\ref{momentum}, we have
 \begin{equation}\label{ram3}
\begin{aligned}
\|(R(m,I(a)m))^{\ell}\|_{\tilde{L}_{t}^{1}(\dot{B}_{q,1}^{\frac{7}{q}-2})}\lesssim \mathcal{X}^{3}.
\end{aligned}
\end{equation}
Boosting (\ref{iam6}), (\ref{iam10}), (\ref{iam8}),  (\ref{iam9}) and (\ref{ram3}), we have
\begin{equation*}
\begin{aligned}
\|(\mathrm{div}((I(a)m)\otimes m))^{\ell}\|_{\tilde{L}_{t}^{1}(\dot{B}_{q,1}^{\frac{7}{q}-3})}
\lesssim \mathcal{X}^{3}.
\end{aligned}
\end{equation*}

\textbf{Step 3: To estimate $\|(h_{3}(a,m))^{\ell}\|_{\tilde{L}_{t}^{1}(\dot{B}_{q,1}^{\frac{7}{q}-2})}$.} For $2\leq q,\, \, \frac{6}{p}+ \frac{2}{q}>\frac{3}{p}+ \frac{2}{q}>1$, with the help of Lemma \ref{linestimate} and Lemma \ref{lem key}, there holds
\begin{equation} \label{ga1}
\begin{aligned}
&\|(T_{G(a)} a^\ell)^{\ell}\|_{\tilde{L}_{t}^{1}(\dot{B}_{q,1}^{\frac{7}{q}-2})}
\lesssim\|G(a)\|_{\tilde{L}_{t}^{2}(\dot{B}_{\infty,1}^{\frac{2}{q}-1})}
\|a^\ell\|_{\tilde{L}_{t}^{2}(\dot{B}_{q,1}^{\frac{5}{q}-1})}\\
&\lesssim\|G(a)\|_{\tilde{L}_{t}^{2}(\dot{B}_{p,1}^{\frac{3}{p}+\frac{2}{q}-1})}
\|a^\ell\|_{\tilde{L}_{t}^{2}(\dot{B}_{q,1}^{\frac{5}{q}-1})}\\
&\lesssim(1+\|a\|_{\tilde{L}_{t}^{\infty}(\dot{B}_{p,1}^{\frac{3}{p}})})
\|a\|_{\tilde{L}_{t}^{2}(\dot{B}_{p,1}^{\frac{3}{p}+\frac{2}{q}-1})}
\|a^\ell\|_{\tilde{L}_{t}^{2}(\dot{B}_{q,1}^{\frac{5}{q}-1})}\\
&\lesssim(1+\|a\|_{\tilde{L}_{t}^{\infty}(\dot{B}_{p,1}^{\frac{3}{p}})})
(\|a^{h}\|_{\tilde{L}_{t}^{2}(\dot{B}_{p,1}^{\frac{3}{p}})}+
\|a^{\ell}\|_{\tilde{L}_{t}^{2}(\dot{B}_{q,1}^{\frac{5}{q}-1})})
\|a^\ell\|_{\tilde{L}_{t}^{2}(\dot{B}_{q,1}^{\frac{5}{q}-1})}\lesssim (1+\mathcal{X})\mathcal{X}^{2}.
\end{aligned}
\end{equation}
On the other hand, with the help of Lemma \ref{linestimate} and Lemma \ref{lem key}, for $2\leq p<6,\,\, 2\leq q\leq p \leq 2q, \,\,\frac{3}{q} - \frac{3}{p} \leq 1$, it holds
\begin{equation} \label{ga2}
\begin{aligned}
&\|(T_{G(a)} a^h)^{\ell}\|_{\tilde{L}_{t}^{1}(\dot{B}_{q,1}^{\frac{7}{q}-2})}
\lesssim\|(T_{G(a)} a^h)^{\ell}\|_{\tilde{L}_{t}^{1}(\dot{B}_{q,1}^{\frac{3}{q}-2})}
\lesssim \|G(a)\|_{\tilde{L}_{t}^{\infty}(\dot{B}_{p,1}^{\frac{3}{p}-1})}
\|a^h\|_{\tilde{L}_{t}^{1}(\dot{B}_{p,1}^{-1+\frac{3}{p}})}
\\
&\lesssim(1+\|a\|_{\tilde{L}_{t}^{\infty}(\dot{B}_{p,1}^{\frac{3}{p}})})\|a\|_{\tilde{L}_{t}^{\infty}(\dot{B}_{p,1}^{\frac{3}{p}-1})}
\|a^h\|_{\tilde{L}_{t}^{1}(\dot{B}_{p,1}^{\frac{3}{p}})}
\lesssim (1+\mathcal{X})\mathcal{X}^{2}.
\end{aligned}
\end{equation}
By Taylor's formula, 
\[
G(a)= G'(0)a + g(a)a, \,  \,  g(0)=0
\]
and given $2\leq p<6,\,\, 2\leq q\leq p \leq 2q, \,\,\frac{3}{q} - \frac{3}{p} \leq 1$,
by Lemma \ref{lem key} and Lemma \ref{linestimate}, similar to \eqref{ha} and \eqref{ia}, we have
\begin{equation} \label{ga3}
\begin{aligned}
&\|(G(a))^{\ell}\|_{\tilde{L}_{t}^{2}(\dot{B}_{q,1}^{\frac{5}{q}-1})}
\lesssim (1+\mathcal{X}+\mathcal{X}^{2})\mathcal{X}.
\end{aligned}
\end{equation}
With aid of Remark~\ref{remarkbon} and (\ref{ga3}), we have
\begin{equation} \label{ga4}
\begin{aligned}
&\|(T_{a} (G(a))^\ell)^{\ell}\|_{\tilde{L}_{t}^{1}(\dot{B}_{q,1}^{\frac{7}{q}-2})}
\lesssim\|a\|_{\tilde{L}_{t}^{2}(\dot{B}_{\infty,1}^{\frac{2}{q}-1})}
\|(G(a))^\ell\|_{\tilde{L}_{t}^{2}(\dot{B}_{q,1}^{\frac{5}{q}-1})}\\
&\lesssim(\|a^{\ell}\|_{\tilde{L}_{t}^{2}(\dot{B}_{q,1}^{\frac{5}{q}-1})}
+\|a^{h}\|_{\tilde{L}_{t}^{2}(\dot{B}_{p,1}^{\frac{2}{q}+\frac{3}{p}-1})})
\|(G(a))^\ell\|_{\tilde{L}_{t}^{2}(\dot{B}_{q,1}^{\frac{5}{q}-1})}\\
&\lesssim(\|a^{\ell}\|_{\tilde{L}_{t}^{2}(\dot{B}_{q,1}^{\frac{5}{q}-1})}
+\|a^{h}\|_{\tilde{L}_{t}^{2}(\dot{B}_{p,1}^{\frac{3}{p}})})
\|(G(a))^\ell\|_{\tilde{L}_{t}^{2}(\dot{B}_{q,1}^{\frac{5}{q}-1})}\\
&\lesssim (1+\mathcal{X}+\mathcal{X}^{2})\mathcal{X}^{2}.
\end{aligned}
\end{equation}
Since
$$(T_{a} G(a)^h)^{\ell} =
 [S_{k_0}, T_{a}] (G(a))^h + 
T_{a} S_{k_0} (G(a))^h.$$
For $2\leq q\leq p \leq 2q$,  with aid of Lemma \ref{commo} and Lemma \ref{linestimate}, we get
\begin{equation} \label{ga5}
\begin{aligned}
&\|[S_{k_0}, T_{a}] (G(a))^h\|_{\tilde{L}_{t}^{1}(\dot{B}_{q,1}^{\frac{7}{q}-2})}
\lesssim\|\nabla a\|_{\tilde{L}_{t}^{\infty}(\dot{B}_{p,1}^{\frac{3}{p}-1})}
\|(G(a))^h\|_{\tilde{L}_{t}^{1}(\dot{B}_{p,1}^{\frac{3}{p}+\frac{4}{q}-2})}
\\
&\lesssim\|\nabla a\|_{\tilde{L}_{t}^{\infty}(\dot{B}_{p,1}^{\frac{3}{p}-1})}
\|(G(a))^h\|_{\tilde{L}_{t}^{1}(\dot{B}_{p,1}^{\frac{3}{p}})}
\\
&\lesssim (1+\|a\|_{\tilde{L}_{t}^{\infty}(\dot{B}_{p,1}^{\frac{3}{p}})})\|a\|_{\tilde{L}_{t}^{\infty}(\dot{B}_{p,1}^{\frac{3}{p}})}
\|a\|_{\tilde{L}_{t}^{1}(\dot{B}_{p,1}^{\frac{3}{p}})}
\lesssim (1+\mathcal{X})\mathcal{X}^{2}.
\end{aligned}
\end{equation}
 Similar to (\ref{ga4}), we have
\begin{equation} \label{ga6}
\begin{aligned}
\|T_{a} S_{k_0} (G(a))^h\|_{\tilde{L}_{t}^{1}(\dot{B}_{q,1}^{\frac{7}{q}-2})}
\lesssim (1+\mathcal{X}+\mathcal{X}^{2})\mathcal{X}^{2}.
\end{aligned}
\end{equation}
Summing up above all terms (\ref{ga4}), (\ref{ga5}) and  (\ref{ga6}), we have
\begin{equation} \label{ga66}
\begin{aligned}
&\|(T_{a} (G(a)))^{\ell}\|_{\tilde{L}_{t}^{1}(\dot{B}_{q,1}^{\frac{7}{q}-2})}
\\
&\lesssim \|(T_{a} (G(a))^{\ell})^{\ell}\|_{\tilde{L}_{t}^{1}(\dot{B}_{q,1}^{\frac{7}{q}-2})}
+\|(T_{a} (G(a))^{h})^{\ell}\|_{\tilde{L}_{t}^{1}(\dot{B}_{q,1}^{\frac{7}{q}-2})}
\lesssim \mathcal{X}^{2}.
\end{aligned}
\end{equation}
 When $\frac{2}{q} + \frac{3}{p}>1$ and $2\leq q\leq p \leq 2q$, Lemma \ref{lem key} and Lemma \ref{linestimate} give  
\begin{equation}\label{ga7}
\begin{aligned}
&\|(R(a^h, G(a)^h))^{\ell}\|_{\tilde{L}_{t}^{1}(\dot{B}_{q,1}^{\frac{7}{q}-2})}\\
&\lesssim\|a^h\|_{\tilde{L}_{t}^{2}(\dot{B}_{p,1}^{\frac{3}{p}})}
\|(G(a))^h\|_{\tilde{L}_{t}^{2}(\dot{B}_{p,1}^{\frac{3}{p} + \frac{4}{q} - 2})} \\
&\lesssim\|a^h\|_{\tilde{L}_{t}^{2}(\dot{B}_{p,1}^{\frac{3}{p}})}
\|(G(a))^h\|_{\tilde{L}_{t}^{2}(\dot{B}_{p,1}^{\frac{3}{p}})} \\
&\lesssim (1+\|a\|_{\tilde{L}_{t}^{\infty}(\dot{B}_{p,1}^{\frac{3}{p}})})\|a^h\|_{\tilde{L}_{t}^{2}(\dot{B}_{p,1}^{\frac{3}{p}})}
\|a\|_{\tilde{L}_{t}^{2}(\dot{B}_{p,1}^{\frac{3}{p}})}\lesssim (1+\mathcal{X})\mathcal{X}^{2}.
\end{aligned}
\end{equation}
When $\frac{5}{q} - 1> 0$, by Lemma \ref{lem key} and (\ref{ga3}),  we have
\begin{equation}\label{ga8} 
\begin{aligned}
&\|(R(G(a)^\ell,a^\ell))^{\ell}\|_{\tilde{L}_{t}^{1}(\dot{B}_{q,1}^{\frac{7}{q}-2})}\\
&\lesssim\|(G(a))^\ell\|_{\tilde{L}_{t}^{2}(\dot{B}_{q,1}^{\frac{5}{q}-1})}
\|a^\ell\|_{\tilde{L}_{t}^{2}(\dot{B}_{q,1}^{-1+\frac{5}{q}})}
\lesssim (1+\mathcal{X}+\mathcal{X}^{2})\mathcal{X}^{2}.
\end{aligned}
\end{equation}
Summing up (\ref{ga7}) and (\ref{ga8}), we have
\begin{eqnarray}\label{ga9}
\|(R(G(a),a))^{\ell}\|_{\tilde{L}_{t}^{1}(\dot{B}_{q,1}^{\frac{7}{q}-2})}
\lesssim \mathcal{X}^{2}.
\end{eqnarray}
Summing up  (\ref{ga1}), (\ref{ga2}), (\ref{ga66}) and (\ref{ga9}), we finish the pressure term
\begin{equation*}
\begin{aligned}
\|(\nabla(G(a)a))^{\ell}\|_{\tilde{L}_{t}^{1}(\dot{B}_{q,1}^{\frac{7}{q}-3})}
\lesssim\|(G(a)a)^{\ell}\|_{\tilde{L}_{t}^{1}(\dot{B}_{q,1}^{\frac{7}{q}-2})}
\lesssim \mathcal{X}^{2}.
\end{aligned}
\end{equation*}

\textbf{Step 4: 
To estimate  $\|(h_2(a, m))^{\ell}\|_{\tilde{L}_{t}^{1}(\dot{B}_{q,1}^{\frac{7}{q}-3})}$.} 
We first consider the quadratic viscous term
\begin{equation} \label{iaml1}
\begin{aligned}
\|( \Delta (I(a) m))^{\ell}\|_{\tilde{L}_{t}^{1}(\dot{B}_{q,1}^{\frac{7}{q}-3})}
\lesssim\|(I(a) m)^{\ell}\|_{\tilde{L}_{t}^{1}(\dot{B}_{q,1}^{\frac{7}{q}-1})}.
\end{aligned}
\end{equation}
For $2\leq q\leq p\leq 2q$ and $\frac{6}{p}+\frac{2}{q}>\frac{3}{p}+\frac{2}{q}>1$, with aid of Remark~\ref{remarkbon} and Lemma~\ref{linestimate}, we have
\begin{equation} \label{iaml2}
\begin{aligned}
&\|(T_{I(a)} m^\ell)^{\ell}\|_{\tilde{L}_{t}^{1}(\dot{B}_{q,1}^{\frac{7}{q}-1})}
\lesssim\|(T_{I(a)} m^\ell)^{\ell}\|_{\tilde{L}_{t}^{1}(\dot{B}_{q,1}^{\frac{7}{q}-2})}
\lesssim\|I(a)\|_{\tilde{L}_{t}^{2}(\dot{B}_{\infty,1}^{\frac{2}{q}-1})}
\|m^\ell\|_{\tilde{L}_{t}^{2}(\dot{B}_{q,1}^{\frac{5}{q}-1})}\\
&\lesssim\|I(a)\|_{\tilde{L}_{t}^{2}(\dot{B}_{p,1}^{\frac{3}{p}+\frac{2}{q}-1})}
\|m^\ell\|_{\tilde{L}_{t}^{2}(\dot{B}_{q,1}^{\frac{5}{q}-1})}\\
&\lesssim(1+\|a\|_{\tilde{L}_{t}^{\infty}(\dot{B}_{p,1}^{\frac{3}{p}})})
\|a\|_{\tilde{L}_{t}^{2}(\dot{B}_{p,1}^{\frac{3}{p}+\frac{2}{q}-1})}
\|m^\ell\|_{\tilde{L}_{t}^{2}(\dot{B}_{q,1}^{\frac{5}{q}-1})}\\
&\lesssim(1+\|a\|_{\tilde{L}_{t}^{\infty}(\dot{B}_{p,1}^{\frac{3}{p}})})( \|a^{h}\|_{\tilde{L}_{t}^{2}(\dot{B}_{p,1}^{\frac{3}{p}})} 
+ \|a^{\ell}\|_{\tilde{L}_{t}^{2}(\dot{B}_{q,1}^{\frac{5}{q}-1})} )
\|m^\ell\|_{\tilde{L}_{t}^{2}(\dot{B}_{q,1}^{\frac{5}{q}-1})}
\\
&\lesssim (1+\mathcal{X})\mathcal{X}\mathcal{Y}.
\end{aligned}
\end{equation}
Using  $$(T_{I(a)} m^h )^\ell= [S_{k_0}, T_{I(a)}] m^h +
 T_{I(a)} S_{k_0} m^h.$$ For 
 $2 \leq q \leq p \leq 2q$, using (\ref{m1}), Lemma~\ref{commo} and Lemma~\ref{linestimate}, we have
\begin{equation} \label{iaml3}
\begin{aligned}
&\|[S_{k_{0}},T_{I(a)}]m^h\|_{\tilde{L}_{t}^{1}(\dot{B}_{q,1}^{\frac{7}{q}-1})}
\leq\|[S_{k_{0}},T_{I(a)}]m^h\|_{\tilde{L}_{t}^{1}(\dot{B}_{q,1}^{\frac{7}{q}-2})}\\
&\lesssim\|\nabla I(a)\|_{\tilde{L}_{t}^{2}(\dot{B}_{p,1}^{\frac{3}{p}-1})}
\|m^h\|_{\tilde{L}_{t}^{2}(\dot{B}_{p,1}^{\frac{3}{p}+\frac{4}{q}-2})}\\
&\lesssim (1+\|a\|_{\tilde{L}_{t}^{\infty}(\dot{B}_{p,1}^{\frac{3}{p}})})
\|a\|_{\tilde{L}_{t}^{2}(\dot{B}_{p,1}^{\frac{3}{p}})}
\|m\|_{\tilde{L}_{t}^{2}(\dot{B}_{p,1}^{\frac{3}{p}})}
\lesssim(1+\mathcal{X})\mathcal{X}^{2}.
\end{aligned}
\end{equation}
 Since 
$\|S_{k_0} m^h\|_{\tilde{L}_{t}^{1}(\dot{B}_{q,1}^{\frac{3}{p}+\frac{4}{q}})}$
 is a sum over finite terms, by Lemma~\ref{lem key} and Lemma~\ref{linestimate}, we have
\begin{equation} \label{iaml4}
\begin{aligned}
&\|T_{I(a)} S_{k_0} m^h\|_{\tilde{L}_{t}^{1}(\dot{B}_{q,1}^{\frac{7}{q}-1})}
 \lesssim \|I(a)\|_{\tilde{L}_{t}^{\infty}(\dot{B}_{p,1}^{-1+\frac{3}{p}})}
\|S_{k_0} m^h\|_{\tilde{L}_{t}^{1}(\dot{B}_{p,1}^{\frac{3}{p}+\frac{4}{q}})}\\
&\lesssim\|I(a)\|_{\tilde{L}_{t}^{\infty}(\dot{B}_{p,1}^{-1+\frac{3}{p}})}
\|m^h\|_{\tilde{L}_{t}^{1}(\dot{B}_{p,1}^{\frac{3}{p}})}
\lesssim(1+\mathcal{X})\mathcal{X}\mathcal{Y}.
\end{aligned}
\end{equation}
Under the conditions $\frac{3}{q} - \frac{3}{p} 
\leq 1$ and $ 2\leq q \leq p \leq 2q$, Lemma~\ref{lem key} and
Lemma~\ref{linestimate} imply
\begin{equation} \label{iaml5}
\begin{aligned}
&\|(T_{m}I(a))^{\ell}\|_{\tilde{L}_{t}^{1}(\dot{B}_{q,1}^{\frac{7}{q}-1})}
\lesssim\|(T_{m}I(a))^{\ell}\|_{\tilde{L}_{t}^{1}(\dot{B}_{q,1}^{\frac{3}{q}-1})}
\\
&\lesssim\|m\|_{\tilde{L}_{t}^{\infty}(\dot{B}_{p,1}^{\frac{3}{p}-1})}
\|I(a)\|_{\tilde{L}_{t}^{1}(\dot{B}_{p,1}^{\frac{3}{p}})}\\
&\lesssim\|m\|_{\tilde{L}_{t}^{\infty}(\dot{B}_{p,1}^{\frac{3}{p}-1})}
(1+\|a\|_{\tilde{L}_{t}^{\infty}(\dot{B}_{p,1}^{\frac{3}{p}})})
\|a\|_{\tilde{L}_{t}^{1}(\dot{B}_{p,1}^{\frac{3}{p}})}
\lesssim (1+\mathcal{X})\mathcal{X}\mathcal{Y}.
\end{aligned}
\end{equation}
When $\frac{2}{q} + \frac{3}{p} >1$, using (\ref{m1}), Lemma~\ref{lem key}
 and Lemma \ref{linestimate}, we have
\begin{equation}\label{iaml6}
\begin{aligned}
&\|(R(m^h, I(a)^h))^{\ell}\|_{\tilde{L}_{t}^{1}(\dot{B}_{q,1}^{\frac{7}{q}-2})} \\
&\lesssim\|m^h\|_{\tilde{L}_{t}^{2}(\dot{B}_{p,1}^{\frac{3}{p}})}
\|(I(a))^h\|_{\tilde{L}_{t}^{2}(\dot{B}_{p,1}^{\frac{3}{p} + \frac{4}{q} - 2})} \\
&\lesssim\|m^h\|_{\tilde{L}_{t}^{2}(\dot{B}_{p,1}^{\frac{3}{p}})}
\|(I(a))^h\|_{\tilde{L}_{t}^{2}(\dot{B}_{p,1}^{\frac{3}{p}})} \\
&\lesssim\|m^h\|_{\tilde{L}_{t}^{2}(\dot{B}_{p,1}^{\frac{3}{p}})}
(1+\|a\|_{\tilde{L}_{t}^{\infty}(\dot{B}_{p,1}^{\frac{3}{p}})})
\|a\|_{\tilde{L}_{t}^{2}(\dot{B}_{p,1}^{\frac{3}{p}})}\lesssim 
(1+\mathcal{X})\mathcal{X}^{2}.
\end{aligned}
\end{equation}
Similar to \eqref{ha} and \eqref{ia}, we have 
 $$\|(I(a))^{\ell}
\|_{\tilde{L}_t^2(\dot{B}_{q,1}^{-1+\frac{5}{q}})}\lesssim 
(1+\mathcal{X}+\mathcal{X}^{2})\mathcal{X},$$
 from which and Lemma~\ref{lem key}, for $q<5$ we have
\begin{equation}\label{iaml7}
\begin{aligned}
\|(R(m^\ell,I(a)^\ell))^{\ell}\|_{\tilde{L}_{t}^{1}(\dot{B}_{q,1}^{\frac{7}{q}-2})}
&\lesssim\|(I(a))^\ell\|_{\tilde{L}_{t}^{2}(\dot{B}_{q,1}^{\frac{5}{q}-1})}
\|m^\ell\|_{\tilde{L}_{t}^{2}(\dot{B}_{q,1}^{-1+\frac{5}{q}})}\\
&\lesssim
(1+\mathcal{X}+\mathcal{X}^{2})\mathcal{X}\mathcal{Y}.
\end{aligned}
\end{equation}
Summing up (\ref{iaml2})--(\ref{iaml7}) and using Proposition~\ref{momentum}, we conclude
\begin{equation*}
\begin{aligned}
\|(-\mu \Delta (I(a) m))^{\ell}\|_{\tilde{L}_{t}^{1}(\dot{B}_{q,1}^{\frac{7}{q}-3})}
\lesssim \mathcal{X}^{2}.
\end{aligned}
\end{equation*}
Similarly, we can prove
\begin{equation*}
\begin{aligned}
\|( (\mu + \lambda) \nabla \mathrm{div} (I(a) m))^{\ell}\|_{\tilde{L}_{t}^{1}(\dot{B}_{q,1}^{\frac{7}{q}-3})}
\lesssim \mathcal{X}^{2}.
\end{aligned}
\end{equation*}
Consequently, gathering the results from steps 1--4, we end up 
\begin{equation*}
\begin{aligned}
\|(h(a,m))^{\ell}\|_{L_{t}^{1}(\dot{B}_{q,1}^{\frac{7}{q}-3}}
\lesssim \mathcal{X}^{2}.
\end{aligned}
\end{equation*}
We complete the proof.
\end{proof}

\subsection{The high frequency analysis}
For high frequency estimate, we employ $L^p$ estimates using the effective velocity technique (see \cite{H20111, H20112}). This approach decouples the system and mitigates potential derivative loss.
\begin{prop}\label{highfre} Assume $2 \leq p< 6,\,\,2 \leq q \leq p \leq 2q$, then we have
\begin{equation*}
\begin{aligned}
&\|u^{h}\|_{\tilde{L}_{t}^{\infty}(\dot{B}_{p,1}^{-1+3/p}) \cap 
\tilde{L}_{t}^{1}(\dot{B}_{p,1}^{1+3/p})} + 
\|a^{h}\|_{\tilde{L}_{t}^{\infty}(\dot{B}_{p,1}^{3/p}) \cap 
L_{t}^{1}(\dot{B}_{p,1}^{3/p})} 
\lesssim \mathcal{X}_0+\mathcal{X}^{2}.
\end{aligned}
\end{equation*}
\end{prop}
\begin{proof}
Let $\mathbb{P} := I+ \nabla(-\Delta)^{-1} \mathrm{div}$ and $\mathbb{Q} :=-\nabla(-\Delta)^{-1} \mathrm{div}$ denote the orthogonal projectors onto the spaces of divergence-free and potential vector fields, respectively. We consider the linearized compressible Navier-Stokes system. We obtain
\begin{equation}\label{linearized CNS}
\left\{
\begin{aligned}
& \partial_t a + \operatorname{div} u =f, \\
& \partial_t u - \mathcal{A} u +  \nabla a = g,
\end{aligned}
\right.
\end{equation}
where $\mathcal{A} = \mu\Delta + (\mu + \lambda)\nabla \operatorname{div}$, and $g = \sum_{i=1}^3 g_i$ with
$$ f = -\operatorname{div} (a u),\, g_1 = -u \cdot \nabla u,\, g_2 =-I(a) \mathcal{A}u,\,
g_3 =k(a) \nabla a,
$$
where $1 = 2\mu + \lambda$,
 $G'(a) =\frac{ P'(a+1)}{a+1} $. Let $I(a) = \frac{a}{a+1}$ 
  and $k(a)=G'(a)-G'(0)$ are smooth functions vanishing at the origin.
   Applying the projections $\mathbb{P}$ and $\mathbb{Q}$ to \eqref{linearized CNS} yields
\begin{equation}\label{lincns}
\left\{
\begin{aligned}
\partial_t a + \operatorname{div} \mathbb{Q} u &= f, \\
\partial_t \mathbb{Q} u -  \Delta \mathbb{Q} u +  \nabla a &= \mathbb{Q} g, \\
\partial_t \mathbb{P} u - \mu \Delta \mathbb{P} u &= \mathbb{P} g.
\end{aligned}
\right.
\end{equation}

\textbf{Step 1. Incompressible part}

First, for the incompressible part, we immediately obtain
\begin{equation}\label{puh1}
\|(\mathbb{P} u)^{h}\|_{\tilde{L}_{t}^{\infty}(\dot{B}_{p,1}^{-1+3/p})} +
\|(\mathbb{P} u)^{h}\|_{\tilde{L}_{t}^{1}(\dot{B}_{p,1}^{3/p+1})}
\lesssim \|(\mathbb{P} u_{0})^{h}\|_{\dot{B}_{p,1}^{-1+3/p}} + 
\|(\mathbb{P} g)^{h}\|_{\tilde{L}_{t}^{1}(\dot{B}_{p,1}^{3/p-1})}.
\end{equation}

\textbf{Step 2. Effective Velocity}

Define the effective velocity 
$w := \nabla(-\Delta)^{-1}a + \mathbb{Q}u$. From \eqref{linearized CNS}, we derive
\[
\partial_t w - \Delta w = \nabla(-\Delta)^{-1}(f - \operatorname{div} g) + w - \nabla(-\Delta)^{-1}a.
\]
Applying Lemma \ref{heatm} (restricted to high frequencies) and Bernstein's inequality yields
\begin{equation}
\begin{aligned}
&\|w^{h}\|_{\tilde{L}_{t}^{\infty}(\dot{B}_{p,1}^{-1+3/p})} + 
\|w^{h}\|_{\tilde{L}_{t}^{1}(\dot{B}_{p,1}^{3/p+1})} \\
&\lesssim \|w_0^{h}\|_{\dot{B}_{p,1}^{-1+3/p}} + 
\|((f, \operatorname{div} g))^{h}\|_{\tilde{L}_{t}^{1}(\dot{B}_{p,1}^{3/p-2})} \\
&\quad + 2^{-2k_0} \left( \|w^{h} \|_{\tilde{L}_{t}^{1}(\dot{B}_{p,1}^{1+3/p})}+ 
\|a^{h}\|_{\tilde{L}_{t}^{1}(\dot{B}_{p,1}^{3/p})} \right) \\
&\lesssim \|(\mathbb{Q} u_0)^{h}\|_{\dot{B}_{p,1}^{-1+3/p}} + 
2^{-2k_0}\|a_0^{h}\|_{\dot{B}_{p,1}^{3/p}} \\
&\quad + 2^{-2k_0} \left( \|w^{h}\|_{\tilde{L}_{t}^{1}(\dot{B}_{p,1}^{1+3/p})} + 
\|a^{h}\|_{\tilde{L}_{t}^{1}(\dot{B}_{p,1}^{3/p})} \right) + 
\|(f, \operatorname{div} g)^{h}\|_{\tilde{L}_{t}^{1}(\dot{B}_{p,1}^{3/p-2})}.
\end{aligned}
\end{equation}
For sufficiently large $k_0$, this implies
\begin{equation}\label{wnh1}
\begin{aligned}
&\|w^{h}\|_{\tilde{L}_{t}^{\infty}(\dot{B}_{p,1}^{-1+3/p})} + 
\|w^{h}\|_{\tilde{L}_{t}^{1}(\dot{B}_{p,1}^{3/p+1})} \\
&\lesssim \|u_0^{h}\|_{\dot{B}_{p,1}^{-1+3/p}} + 
2^{-2k_0}\|a_0\|_{\dot{B}_{p,1}^{3/p}}^{h}+ 2^{-2k_0} \|a^{h}\|_{\tilde{L}_{t}^{1}(\dot{B}_{p,1}^{3/p})} + 
\|(f, \operatorname{div} g)^{h}\|_{\tilde{L}_{t}^{1}(\dot{B}_{p,1}^{3/p-2})}.
\end{aligned}
\end{equation}

\textbf{Step 3. Density Damping}

From \eqref{linearized CNS}, we observe that
\begin{equation}\label{aNtrans1}
\partial_t a + u \cdot \nabla a + a = -\operatorname{div} w - a \operatorname{div} u.
\end{equation}
Applying Lemma \ref{transport} (restricted to high frequencies) to \eqref{aNtrans1} and \eqref{wnh1} gives
\EQ{\label{anh1}
&\|a^{h} \|_{\tilde{L}_{t}^{\infty}(\dot{B}_{p,1}^{3/p}) \cap \tilde{L}_{t}^{1}(\dot{B}_{p,1}^{3/p})}\\
&\lesssim \|a_0^{h}\|_{\dot{B}_{p,1}^{3/p}} + 
\|(\operatorname{div} w - a \operatorname{div} u)^h\|_{\tilde{L}_{t}^{1}(\dot{B}_{p,1}^{3/p})}+\int_{0}^{t}
\|\nabla u\|_{\dot{B}_{p,1}^{3/p}}\|a\|_{\dot{B}_{p,1}^{3/p}}d\tau \\
&\lesssim (1 + 2^{-2k_0}) \|a_0^{h}\|_{\dot{B}_{p,1}^{3/p}} + 
\|u_0^h\|_{\dot{B}_{p,1}^{-1+3/p}}+ 2^{-2k_0} \|a^{h}\|_{\tilde{L}_{t}^{1}(\dot{B}_{p,1}^{3/p})} \\
& \ \ + 
\|(a\cdot \Div u)^h\|_{\tilde{L}_{t}^{1}(\dot{B}_{p,1}^{3/p})}
+\|(f, \operatorname{div} g)^{h}\|_{\tilde{L}_{t}^{1}(\dot{B}_{p,1}^{3/p-2})}.
}

We now focus on high frequencies, where the analysis is classical. Specifically,  
Sobolev embedding yields
\[
\|(a,u)\|_{\tilde{L}_{t}^{2}(\dot{B}_{p,1}^{3/p})} \lesssim 
\|(a,u)^h\|_{\tilde{L}_{t}^{2}(\dot{B}_{p,1}^{3/p})} + 
\|(a,u)^\ell\|_{\tilde{L}_{t}^{2}(\dot{B}_{q,1}^{-1+5/q})} \lesssim \mathcal{X},
\]
from which and Lemma~\ref{prodlemma}, we have
\[
\|(a \operatorname{div} u)^{h}\|_{\tilde{L}_{t}^{1}(\dot{B}_{p,1}^{3/p})} \lesssim 
\|a\|_{\tilde{L}_{t}^{\infty}(\dot{B}_{p,1}^{3/p})} \|u\|_{\tilde{L}_{t}^{1}(\dot{B}_{p,1}^{1+3/p})}\lesssim \mathcal{X}^{2}
\]
and
\begin{equation}\label{aulow1}
\begin{aligned}
\|f^{h}\|_{\tilde{L}_{t}^{1}(\dot{B}_{p,1}^{-2+3/p})} \lesssim 
\|(au)^{h}\|_{\tilde{L}_{t}^{1}(\dot{B}_{p,1}^{3/p})} \lesssim
\|u\|_{\tilde{L}_{t}^{2}(\dot{B}_{p,1}^{3/p})}\|a\|_{\tilde{L}_{t}^{2}(\dot{B}_{p,1}^{3/p})}
 \lesssim \mathcal{X}^{2}
\end{aligned}
\end{equation}
and similarly for the velocity,
 $\|a \operatorname{div} u\|_{L_{t}^{1}(\dot{B}_{p,1}^{3/p})}^{h} \lesssim \mathcal{X}^2$. For $q\geq2$ and $2\leq p<6$, using Lemma~\ref{prodlemma}, then we have
\begin{equation}\label{aulow2}
\begin{aligned}
&\| (\operatorname{div} g)^{h}\|_{\tilde{L}_{t}^{1}(\dot{B}_{p,1}^{3/p-2})} \lesssim 
\| g^{h} \|_{\widetilde{L}_{t}^{1}(\dot{B}_{p,1}^{3/p-1})}
 \lesssim
 \|u\|_{\tilde{L}_{t}^{\infty}(\dot{B}_{p,1}^{-1+3/p})}
 \|\nabla u\|_{\tilde{L}_{t}^{1}(\dot{B}_{p,1}^{3/p})}\\
&
 +\|a\|_{\tilde{L}_{t}^{\infty}(\dot{B}_{p,1}^{-1+3/p})}
 \|\nabla^{2} u\|_{\tilde{L}_{t}^{1}(\dot{B}_{p,1}^{-1+3/p})}+\|a\|_{\tilde{L}_{t}^{2}(\dot{B}_{p,1}^{3/p})}
 \|\nabla u\|_{\tilde{L}_{t}^{2}(\dot{B}_{p,1}^{-1+3/p})}\lesssim \mathcal{X}^{2}.
\end{aligned}
\end{equation}
Combining \eqref{puh1}, \eqref{anh1},
\eqref{aulow1} and \eqref{aulow2},
and taking $k_0$ sufficiently large such that $2^{-k_0} \ll 1$, we obtain
\begin{equation}
\begin{aligned}
&\|u^{h}\|_{\tilde{L}_{t}^{\infty}(\dot{B}_{p,1}^{-1+3/p}) \cap 
\tilde{L}_{t}^{1}(\dot{B}_{p,1}^{1+3/p})} + 
\|a^{h}\|_{\tilde{L}_{t}^{\infty}(\dot{B}_{p,1}^{3/p}) \cap 
\tilde{L}_{t}^{1}(\dot{B}_{p,1}^{3/p})} \\
&\lesssim  \|u_0^{h}\|_{\dot{B}_{p,1}^{-1+3/p}} + 
\|a_0\|_{\dot{B}_{p,1}^{3/p}}^{h}+ \|a \operatorname{div} u\|_{\tilde{L}_{t}^{1}(\dot{B}_{p,1}^{3/p})} + 
\|(f, \operatorname{div} g)^{h}\|_{\tilde{L}_{t}^{1}(\dot{B}_{p,1}^{3/p-2})}\\
&\lesssim
\|u_0^{h}\|_{\dot{B}_{p,1}^{-1+3/p}} + 
\|a_0^{h}\|_{\dot{B}_{p,1}^{3/p}}+\mathcal{X}^{2}.
\end{aligned}
\end{equation}
We thus finish the proof of Proposition ~\ref
{highfre}.
\end{proof}

\section{Proof of Theorem~\ref{gwp}}

In this section we prove Theorem~\ref{gwp}. The proof proceeds first deriving the a priori estimates, implementing and completing the bootstrap argument using these estimates. Suppose $(a,u)$ solves \eqref{CNS} with
\begin{equation}\label{a1}
\mathcal{X} \lesssim 1 \quad \text{and} \quad \|a\|_{\widetilde{L}_{t}^{\infty}(\dot{B}_{p,1}^{3/p})} \lesssim 1.
\end{equation}
Under the condition
\begin{equation}\label{x2}
C\mathcal{X} \leq  \frac{1}{2},
\end{equation}
and using Propositions~\ref{lowfre} and~\ref{highfre}, we obtain
\begin{equation*}
\mathcal{X} \leq C\mathcal{X}_{0}+ C\mathcal{X}^2,
\end{equation*}
which yields
\begin{equation}\label{xsmall}
\mathcal{X} \leq C\mathcal{X}_{0}.
\end{equation}
Applying a standard bootstrap argument, we conclude that for sufficiently small initial data \(\mathcal{X}_{0}\), conditions \eqref{a1} and \eqref{x2} persist for all \(t \in [0, T^*)\).
Consequently, the solution $(a,u)$ exists globally with 
uniform bound \eqref{xsmall}, provided \(\mathcal{X}_{0}\) is chosen sufficiently small.

To prove the global existence and uniqueness, we are left to show that $X_{q,p}$ norm (see \eqref{bound}) is locally propagated along the flow. Let $(a,u)\in C_T\X_p$ be the local solution on $[0,T]$ obtained in the local well-posedness. We denote
\begin{align*}
E_p(t)=&\|a\|_{\tilde L^\infty_{T}( \dot{B}_{p,1}^{\frac{3}{p}})}+\| u\|_{\tilde L^\infty_{T}(\dot{B}_{p,1}^{-1+\frac{3}{p}})}+\| (\partial_{t}u,\nabla^2 u)\|_{ \tilde L^1_{T}(\dot{B}_{p,1}^{\frac{3}{p}})}\lesssim 1.
\end{align*}
Therefore, it suffices to show the boundedness of $\|(a,m)^{\ell}\|_{\tilde{L}^{\infty}_{T}(\dot{B}_{q,1}^{-3
+\frac{7}{q}})}$ provided that $\mathcal{X}_{0}$ is bounded. 

First we estimate $\|a^{\ell}\|_{\tilde{L}_{t}^{\infty}(\dot{B}_{p,1}^{-1
+\frac{3}{p}})}$, by the continuity equation and get
\begin{equation}\begin{aligned}\|a^{\ell}\|_{\tilde{L}_{t}^{\infty}(\dot{B}_{p,1}^{-1
+\frac{3}{p}})}&\lesssim
\|a_0\|_{\dot{B}_{p,1}^{-1
+\frac{3}{p}}}^{\ell}+\|\mathrm{div} (au)\|_{\tilde{L}_{t}^{1}(\dot{B}_{p,1}^{-1
+\frac{3}{p}})}^{\ell}\\&
\lesssim\|a_0\|_{\dot{B}_{p,1}^{-1
+\frac{3}{p}}}^{\ell}+T^{\frac 12}
\|a\|_{\tilde{L}_{t}^{\infty}(\dot{B}_{p,1}^{
\frac{3}{p}})}\|u\|_{\tilde{L}_{t}^{2}(\dot{B}_{p,1}^{\frac{3}{p}})}
\lesssim \mathcal{X}_{0}+T^{1/2}E^2_p(t).\end{aligned}
\end{equation}
Next we estimate the momentum using the momentum equation
$$\partial_t m - \mathcal{A} m  =- \nabla a+h(a, m)$$
and obtain
\begin{equation}
\begin{aligned}
&\|m^{\ell}\|_{\tilde{L}_{t}^{\infty}(\dot{B}_{q,1}^{-3
+\frac{7}{q}})\cap \tilde{L}_{t}^{1}(\dot{B}_{q,1}^{-1
+\frac{7}{q}})}\\
&\lesssim
\|m_0^{\ell}\|_{\dot{B}_{q,1}^{-3
+\frac{7}{q}}}+\|(\nabla a)^{\ell}\|_{\tilde{L}_{t}^{1}(\dot{B}_{q,1}^{-3
+\frac{7}{q}})}
+
\|h(a, m)^{\ell}\|_{\tilde{L}_{t}^{1}(\dot{B}_{q,1}^{-3
+\frac{7}{q}})}\\
&\lesssim\|m_0^{\ell}\|_{\dot{B}_{q,1}^{-3
+\frac{7}{q}}}+T2^{k_0}\| a\|^{\ell}_{\tilde{L}_{t}^{\infty}(\dot{B}_{q,1}^{-3
+\frac{7}{q}})}
+\|h(a, m)^{\ell}\|_{\tilde{L}_{t}^{1}(\dot{B}_{q,1}^{-3
+\frac{7}{q}})}.
\end{aligned}
\end{equation}
It remains to estimate $h$. 
For the convection terms, we focus on the term $\mathrm{div}(m\otimes m)$. We perform the Bony's paraproduct decomposition. We get for $q\leq4$ that 
\begin{equation*}
\begin{aligned}
\|(\mathrm{div}(T_m m))^{\ell}\|_{\tilde{L}_{t}^{1}(\dot{B}_{q,1}^{-3
+\frac{7}{q}})}&\lesssim\|(T_m m)^{\ell}\|_{\tilde{L}_{t}^{1}(\dot{B}_{q,1}^{-1
+\frac{3}{q}})}\\
&\lesssim
\|m\|_{\tilde{L}_{t}^{\infty}(\dot{B}_{p,1}^{-1
+\frac{3}{p}})}\|m\|_{\tilde{L}_{t}^{1}(\dot{B}_{p,1}^{\frac{3}{p}})}\lesssim T^{\frac 12} E^2_p(t),
\end{aligned}
\end{equation*}
where we utilize
$$\|m\|_{\tilde{L}_{t}^{\infty}(\dot{B}_{p,1}^{-1
+\frac{3}{p}})}\lesssim C_{\|a\|_{\dot{B}_{p,1}^{\frac{3}{p}}}}\|u\|_{\tilde{L}_{t}^{\infty}(\dot{B}_{p,1}^{-1
+\frac{3}{p}})}\lesssim E_p(t);$$
$$\|m\|_{\tilde{L}_{t}^{1}(\dot{B}_{p,1}^{\frac{3}{p}})}\lesssim T^{\frac 12}C_{\|a\|_{\dot{B}_{p,1}^{\frac{3}{p}}}}\|u\|_{\tilde{L}_{t}^{2}(\dot{B}_{p,1}^{\frac{3}{p}})}\lesssim T^{\frac 12}E_p(t),$$
with $C_{\|a\|_{\dot{B}_{p,1}^{\frac{3}{p}}}}=1+\|a\|_{\dot{B}_{p,1}^{\frac{3}{p}}}$. For the remainder, it holds
\begin{equation*}\begin{aligned}\|(\mathrm{div}(R(m, m)))^{\ell}\|_{\tilde{L}_{t}^{1}(\dot{B}_{q,1}^{-3
+\frac{7}{q}})}&\lesssim\|m\|_{\tilde{L}_{t}^{2}(\dot{B}_{p,1}^{\frac{3}{p}})}\|m\|_{\tilde{L}_{t}^{2}(\dot{B}_{p,1}^{-2
+\frac 4q+\frac{3}{p}})}.
\end{aligned}
\end{equation*}
Note that
$$\|m\|_{\tilde{L}_{t}^{\frac {2q}{4-q}}(\dot{B}_{p,1}^{-2
+\frac 4q+\frac{3}{p}})}\lesssim\|m\|^{2-\frac{4}{q}}_{\tilde{L}_{t}^{\infty}(\dot{B}_{p,1}^{-1
+\frac{3}{p}})}\|m\|^{\frac{4}{q}-1}_{\tilde{L}_{t}^{2}(\dot{B}_{p,1}^{\frac{3}{p}})}\lesssim E_p(t).$$
Thus we naturally get
\begin{equation*}
\begin{aligned}\|(\mathrm{div}(R(m, m)))^{\ell}\|_{\tilde{L}_{t}^{1}(\dot{B}_{q,1}^{-3
+\frac{7}{q}})}
\lesssim& T^{\frac{q-2}{2}}\|m\|_{\tilde{L}_{t}^{2}(\dot{B}_{p,1}^{\frac{3}{p}})}\|m\|_{\tilde{L}_{T}^{\frac {2q}{4-q}}(\dot{B}_{p,1}^{-2
+\frac 4q+\frac{3}{p}})}\lesssim T^{\frac{q-2}{2}} E^2_p(t),
\end{aligned}
\end{equation*}
On the other hand, if $q>4$, we have
\begin{equation*}\begin{aligned}\|(\mathrm{div}(T_m m))^{\ell} \|_{\tilde{L}_{t}^{1}(\dot{B}_{q,1}^{-3
+\frac{7}{q}})}&\lesssim
\|m\|_{\tilde{L}_{t}^{\infty}(\dot{B}_{p,1}^{-1
+\frac{3}{p}})}\|m\|_{\tilde{L}_{t}^{1}(\dot{B}_{p,1}^{-1
+\frac 4q+\frac{3}{p}})}\\&\lesssim T^{\frac{q-2}{q}}
\|m\|_{\tilde{L}_{t}^{\infty}(\dot{B}_{p,1}^{-1
+\frac{3}{p}})}\|m\|_{\tilde{L}_{T}^{\frac q2}(\dot{B}_{p,1}^{-1
+\frac 4q+\frac{3}{p}})}\lesssim T^{\frac{q-2}{q}}E^2_p(t),
\end{aligned}
\end{equation*}
where we used the following interpolation:
$$\|m\|_{\tilde{L}_{t}^{\frac q2}(\dot{B}_{p,1}^{-1
+\frac 4q+\frac{3}{p}})}\lesssim\|m\|^{1-\frac{4}{q}}_{\tilde{L}_{t}^{\infty}(\dot{B}_{p,1}^{-1
+\frac{3}{p}})}\|m\|^{\frac{4}{q}}_{\tilde{L}_{t}^{2}(\dot{B}_{p,1}^{\frac{3}{p}})}\lesssim E_p(t).$$
For the remainder, it also holds
\begin{equation*}\begin{aligned}\|(\mathrm{div}(R(m, m)))^{\ell}\|_{\tilde{L}_{t}^{1}(\dot{B}_{q,1}^{-3
+\frac{7}{q}})}&\lesssim\|m\|_{\tilde{L}_{t}^{\infty}(\dot{B}_{p,1}^{-1+\frac{3}{p}})}\|m\|_{\tilde{L}_{t}^{1}(\dot{B}_{p,1}^{-1
+\frac 4q+\frac{3}{p}})}\lesssim T^{\frac{q-2}{2}} E^2_p(t).
\end{aligned}
\end{equation*}
For the other term $\Div(I(a)m\otimes m)$, the proof is similar. This completes estimates for $h_1$.

For the viscous term, Bony decomposition and Lemma imply for $q\leq4$
\begin{equation*}
\begin{aligned}\|(\Delta(T_{I(a)}m))^{\ell}\|_{\tilde{L}_{t}^{1}(\dot{B}_{q,1}^{-3
+\frac{7}{q}})}&\lesssim\|(T_{I(a)}m)^{\ell}\|_{\tilde{L}_{t}^{1}(\dot{B}_{q,1}^{-1
+\frac{3}{q}})}\\
&\lesssim
\|I(a)\|_{\tilde{L}_{t}^{\infty}(\dot{B}_{p,1}^{-1
+\frac{3}{p}})}\|m\|_{\tilde{L}_{t}^{1}(\dot{B}_{p,1}^{\frac{3}{p}})}\lesssim T^{\frac 12} E^2_p(t),
\end{aligned}
\end{equation*}
\begin{equation*}\begin{aligned}\|(\Delta(T_{m}I(a)))^{\ell}\|_{\tilde{L}_{t}^{1}(\dot{B}_{q,1}^{-3
+\frac{7}{q}})}&\lesssim
\|m\|_{\tilde{L}_{t}^{\infty}(\dot{B}_{p,1}^{-1
+\frac{3}{p}})}\|I(a)\|_{\tilde{L}_{t}^{1}(\dot{B}_{p,1}^{\frac{3}{p}})}\lesssim T E^2_p(t);
\end{aligned}
\end{equation*}
\begin{equation*}\begin{aligned}\|(\Delta(R(I(a),m)))^{\ell}\|_{\tilde{L}_{t}^{1}(\dot{B}_{q,1}^{-3
+\frac{7}{q}})}&\lesssim
\|I(a)\|_{\tilde{L}_{t}^{2}(\dot{B}_{p,1}^{\frac{3}{p}})}\|m\|_{\tilde{L}_{t}^{2}(\dot{B}_{p,1}^{-2
+\frac 4q+\frac{3}{p}})}\lesssim T^{\frac{q-1}{2}} E^2_p(t).
\end{aligned}
\end{equation*}
Symmetrically we can handle with $q>4$ and this finish the viscous terms. The pressure term can be controlled  by parallel estimates as convection term where
\begin{equation}
\begin{aligned}
\|(\nabla(G(a)a)^{\ell}\|_{\tilde{L}_{t}^{1}(\dot{B}_{q,1}^{-3
+\frac{7}{q}})}&\lesssim\|a\|_{\tilde{L}_{t}^{\infty}(\dot{B}_{p,1}^{-1
+\frac{3}{p}})}\|a\|_{\tilde{L}_{t}^{1}(\dot{B}_{p,1}^{\frac{3}{p}})}\\
&+
\|G(a)\|_{\tilde{L}_{t}^{2}(\dot{B}_{p,1}^{\frac{3}{p}})}\|a\|_{\tilde{L}_{t}^{2}(\dot{B}_{p,1}^{-2
+\frac 4q+\frac{3}{p}})}\lesssim T^{\frac{q-2}{q}}E^2_p(t)
\end{aligned}
\end{equation}
and we conclude for $t\in[0,T]$,
$$\|(h(a, m))^{\ell}\|_{\tilde{L}_{t}^{1}(\dot{B}_{q,1}^{-3
+\frac{7}{q}})}\leq C_T( E^2_p(t)+E^3_p(t))$$
and this concludes the boundedness of $m$ with
\begin{equation}\begin{aligned}\|m^{\ell}\|_{\tilde{L}_{t}^{\infty}(\dot{B}_{q,1}^{-3
+\frac{7}{q}})\cap \tilde{L}_{t}^{1}(\dot{B}_{q,1}^{-1
+\frac{7}{q}})}
\lesssim\|m_0^{\ell}\|_{\dot{B}_{q,1}^{-3
+\frac{7}{q}}}+T2^{k_0}\| a^{\ell}\|_{\tilde{L}_{t}^{\infty}(\dot{B}_{q,1}^{-3
+\frac{7}{q}})}
+E^2_p(t)+E^3_p(t).
\end{aligned}
\end{equation}
On the other hand, for the density, the continuity equation yields
\begin{equation*}\begin{aligned}\|a^{\ell}\|_{\tilde{L}_{t}^{\infty}(\dot{B}_{q,1}^{-3
+\frac{7}{q}})}\lesssim&
\|a_0\|_{\dot{B}_{q,1}^{-3
+\frac{7}{q}}}^{\ell}+\|\mathrm{div} m^{\ell}\|_{\tilde{L}_{t}^{1}(\dot{B}_{q,1}^{-3
+\frac{7}{q}})}\\
\lesssim&\|a_0^{\ell}\|_{\dot{B}_{q,1}^{-3
+\frac{7}{q}}}+T2^{k_0}
\|m^{\ell}\|_{\tilde{L}_{t}^{\infty}(\dot{B}_{q,1}^{-3
+\frac{7}{q}})}.
\end{aligned}
\end{equation*}
Above two inequaities finally yields
\begin{equation*}\begin{aligned}\|(a,m)^{\ell}\|_{\tilde{L}_{t}^{\infty}(\dot{B}_{q,1}^{-3
+\frac{7}{q}})}
\lesssim\|(a_0,m_0)^{\ell}\|_{\tilde{L}_{t}^{\infty}(\dot{B}_{q,1}^{-3
+\frac{7}{q}})}+T2^{k_0}
\|(a,m)^{\ell}\|_{\tilde{L}_{t}^{\infty}(\dot{B}_{q,1}^{-3
+\frac{7}{q}})}+E^2_p(t)+E^3_p(t).
\end{aligned}
\end{equation*}
Then we are able to prove $\|(a,m)^{\ell}\|_{\tilde{L}_{t}^{\infty}(\dot{B}_{q,1}^{-3
+\frac{7}{q}})}$ is also locally bounded once we take $T$ to be sufficient small such that $T2^{k_0}\ll1$. Therefore we could apply the local results on uniform estimates to get to Theorem \ref{gwp}.

Finally we show the continuity. According to \cite{GYZ2024}, for any initial data \((a_0, u_0) \in \X_p\), 
there exist a neighbourhood 
\(U \subset \X_p\) of \((a_0, u_0)\) and a time \(T = T(U) > 0\) such that for every
 \((\tilde{a}_0, \tilde{u}_0) \in U\), the Cauchy problem \eqref{CNS} has a unique solution  
\[
(\tilde{a}, \tilde{u}) := S_T(\tilde{a}_0, \tilde{u}_0) \in Z_p(T) = C\big([0,T]; 
\dot{B}^{\frac{3}{p}}_{p,1}\big) \times C\big([0,T]; \dot{B}^{-1+\frac{3}{p}}_{p,1}\big),
\]
and the solution map \(S_T: U \to Z_p(T)\) is continuous. Since the solution \((a, u)\) satisfies
\[
\|a\|_{L^\infty_{T}(\dot{B}^{\frac{3}{p}}_{p,1})} + \|u\|_{L^\infty_{T}(\dot{B}^{-1+\frac{3}{p}}_{p,1})} \lesssim \|(a,u)\|_{X_{q,p}} \lesssim \mathcal{X}_0,
\]
and \(\mathcal{X}_0\) is sufficiently small, a bootstrap argument yields that
 \(S_T\) is continuous from \(U\) into \(Z_p(T)\) for any \(T > 0\).

\appendix
\section{Some paraproduct estimate}
Some basic lemmas concerning paraproducts are presented below. We begin by introducing the following nonclassical $L^{q}$--$L^{p}$ type estimates.estimates:
\begin{lemma}\label{lem key}
Let $s,m,m_{1},m_{2}\in \mathbb{R}$ and $1 \leq p,q,r\leq\infty$. Then there holds for $m=m_{1}+m_{2}$ and $1=1/p+1/p'$ such that
\begin{equation}\label{para}
\|T_{a}b\|_{\dot{B}^{s-m+\frac{d}{q}-\frac{d}{p}}_{q,r}}\lesssim
\|a\|_{\dot{B}^{\frac{d}{p}-m_{1}}_{p,1}}\|b\|_{\dot{B}^{s-m_{2}}_{p,r}} 
\,\,\,\,\mathrm{if}\,\,\,\,
m_{1}\geq d\max(0,\frac{1}{q}-\frac{1}{p})\, \mbox{and} \, \, p\leq2q;
\end{equation}
and
\begin{equation}\label{remains}
\|R(a,b)\|_{\dot{B}^{s-m+\frac{d}{q}-\frac{d}{p}}_{q,r}}
\lesssim\|a\|_{\dot{B}^{\frac{d}{p}-m_{1}}_{p,1}}
\|b\|_{\dot{B}^{s-m_{2}}_{p,r}}\,\,\,\,\mathrm{if}\,\,\,\,
s>m-d\min(\frac{1}{p},\frac{1}{p'})\, \mbox{and} \, \,  p\leq2q.
\end{equation}
\end{lemma}

\begin{proof}
To prove (\ref{para}), it firstly follows from the definition of $T_{a}b$ and the spectral cut-off property that
$$\dot{\Delta}_{k}T_{a}b=\dot{\Delta}_{k}\Big(\sum_{k'}\dot{S}_{k'-1}a\dot{\Delta}_{k'}b\Big)=\sum_{|k-k'|\leq 4}\dot{\Delta}_{k}(\dot{S}_{k'-1}a\dot{\Delta}_{k'}b).$$

On the one hand, we assume that $q\leq p$. Set $\frac{1}{q}=\frac{1}{\tilde p}+\frac{1}{p}$. The assumption $p\leq2q$ implies that $p\leq \tilde p$. Hence, we have
\EQN{
\|\dot{\Delta}_{k}T_{a}b\|_{L^q}\lesssim &\sum_{|k-k'|\leq 4}\|\dot{S}_{k'-1}a\dot{\Delta}_{k'}b\|_{L^q}\\
\lesssim &\sum_{|k-k'|\leq 4}\sum_{l\leq k'-2} \|\dot{\Delta}_{l}a\|_{L^{\tilde p}}\|\dot{\Delta}_{k'}b\|_{L^p}\\
\lesssim &\sum_{|k-k'|\leq 4}\big(\sum_{l\leq k'-2}2^{(\frac{d}{p}-\frac{d}{q}+m_{1})l}2^{(\frac{d}{p}-m_{1})l}\|\dot{\Delta}_{l}a\|_{L^p}\big)\|\dot{\Delta}_{k'}b\|_{L^p}.
}
Note that $\frac{d}{p}-\frac{d}{q}+m_{1}\geq 0$, we deduct that
\EQN{
\|\dot{\Delta}_{k}T_{a}b\|_{L^q}
\lesssim&\sum_{|k-k'|\leq 4}\big(\sup_{l\leq k'-2}2^{(\frac{d}{p}-\frac{d}{q}+m_{1})l}\big)\|a\|_{\dot{B}^{\frac{d}{p}-m_{1}}_{p,1}}\|\dot{\Delta}_{k'}b\|_{L^p}\\
\lesssim&\sum_{|k-k'|\leq 4}2^{(\frac{d}{p}-\frac{d}{q}+m_{1})k'}\|\dot{\Delta}_{k'}b\|_{L^p}\|a\|_{\dot{B}^{\frac{d}{p}-m_{1}}_{p,1}}.
}
On the other hand, if $q>p$ then
\EQN{
\|\dot{\Delta}_{k}T_{a}b\|_{L^q}\lesssim&\sum_{|k-k'|\leq 4}\|\dot{S}_{k'-1}a\dot{\Delta}_{k'}b\|_{L^q}\lesssim\sum_{|k-k'|\leq 4}\sum_{l\leq k'-2}\|\dot{\Delta}_{l}a\|_{L^\infty}\|\dot{\Delta}_{k'}b\|_{L^q}\\
\lesssim&\sum_{|k-k'|\leq 4}2^{(\frac{d}{p}-\frac{d}{q})k'}\big(\sup_{l\leq k'-2}2^{m_{1}l}\big)\|a\|_{\dot{B}^{\frac{d}{p}-m_{1}}_{p,1}}\|\dot{\Delta}_{k'}b\|_{L^p}\\
\lesssim&\sum_{|k-k'|\leq 4}2^{(\frac{d}{p}-\frac{d}{q}+m_{1})k'}\|\dot{\Delta}_{k'}b\|_{L^p}\|a\|_{\dot{B}^{\frac{d}{p}-m_{1}}_{p,1}},
}
where $m_{1}\geq0$ was used. Consequently, employing Young's inequality enables us to get to (\ref{para}).
We turn to prove (\ref{remains}). By the spectrum cut-off, one has
$$\dot{\Delta}_{k}R(a,b)=\dot{\Delta}_{k}\Big(\sum_{k'}\tilde{\dot{\Delta}}_{k'}a\dot{\Delta}_{k'}\Big)=\sum_{k\leq k'+2}\dot{\Delta}_{k}(\tilde{\dot{\Delta}}_{k'}a\dot{\Delta}_{k'}b),$$
where $\tilde{\dot{\Delta}}_{k'}a\triangleq\sum_{|k-k'|\leq1}\dot{\Delta}_{k}a.$ We consider the case $1\leq p\leq2$ first. By H\"{o}lder and Bernstein inequalities, we arrive at
\EQN{
\|\dot{\Delta}_{k}R(a,b)\|_{L^q}
\lesssim&2^{(d-\frac{d}{q})k}\|\dot{\Delta}_{k}\Big(\sum_{k'}\tilde{\dot{\Delta}}_{k'}a\dot{\Delta}_{k'}b\Big)\|_{L^1}\lesssim2^{(d-\frac{d}{q})k}\sum_{k'\geq k-2}\|\dot{\Delta}_{k'}a\|_{L^{2}}\|\dot{\Delta}_{k'}b\|_{L^2}\\
\lesssim&2^{(d-\frac{d}{q})k}\sum_{k'\geq k-2}2^{(m-s+\frac{d}{p}-d)k'}2^{(\frac{d}{p}-m_{1})k'}\|\dot{\Delta}_{k'}a\|_{L^{p}}2^{(s-m_{2})k'}\|\dot{\Delta}_{k'}b\|_{L^p}.
}
If $s>m-\frac{d}{p'}$, then $s-m-\frac{d}{p}+d>0$, so one can immediately have (\ref{remain}),
where H$\ddot{o}$lder inequality for series was performed. On the other hand, we deal with the case $2< p\leq2q$. By H\"{o}lder and Bernstein inequalities, we have
\EQN{
\|\dot{\Delta}_{k}R(a,b)\|_{L^q}
\lesssim&2^{(\frac{2d}{p}-\frac{d}{q})k}\|\dot{\Delta}_{k}\Big(\sum_{k'}\tilde{\dot{\Delta}}_{k'}a\dot{\Delta}_{k'}b\Big)\|_{L^\frac{p}{2}}\lesssim2^{(\frac{2d}{p}-\frac{d}{q})k}\sum_{k'\geq k-2}\|\dot{\Delta}_{k'}a\|_{L^{p}}\|\dot{\Delta}_{k'}b\|_{L^p}\\
\lesssim&2^{(\frac{2d}{p}-\frac{d}{q})k}\sum_{k'\geq k-2}2^{(m-s-\frac{d}{p})k'}2^{(\frac{d}{p}-m_{1})k'}\|\dot{\Delta}_{k'}a\|_{L^{p}}2^{(s-m_{2})k'}\|\dot{\Delta}_{k'}b\|_{L^p}.
}
Since $s>m-\frac{d}{p}$, i.e. $s-m+\frac{d}{p}>0$, we similarly get (\ref{remains}).
\end{proof}

\begin{rema}\label{remarkbon}
Let $s, s_1, s_2 \in \mathbb{R}, s=s_1+s_2$  and \(1 \leq p, r \leq \infty\).
 For the paraproduct, when \(s_1 \leq 0\), we have
\begin{equation*}
 \|T_f g\|_{\dot{B}^{s}_{p,r}} \lesssim \|f\|_{\dot{B}^{s_1}_{\infty,1}} \|g\|_{\dot{B}^{s_2}_{p,r}}.   
\end{equation*}
Furthermore, if \(s > 0\), the remainder satisfies
\begin{equation*}\label{remain}
\|R(f, g)\|_{\dot{B}^{s}_{p,r}} \lesssim \|f\|_{\dot{B}^{s_1}_{\infty,1}} \|g\|_{\dot{B}^{s_2}_{p,r}}.
\end{equation*}
Moreover, the case \(s = 0\) holds in the above inequalities when \(r = \infty\).
\end{rema}

Motivated by Lemma 6.1 in  \cite{DH2016}, we now prove the following commutator estimates.
\begin{lemma}\label{commo}
Let \(A(D)\) be a \(0\)-order Fourier multiplier, and \(k_{0} \in \mathbb{Z}\). 
Let $2\leq q\leq p\leq2q$. Then there exists a constant $C$ depends on
 $k_{0}$ such that for $s\leq1$ and $\sigma\in \mathbb{R}$, it holds
\begin{equation*}
\|[\dot S_{k_{0}}A(D),T_{a}]b\|_{\dot{B}_{q,1}^{\sigma+s}}
\leq C\|\nabla a\|_{\dot{B}_{p,1}^{s-1+\frac{2d}{p}-\frac{d}{q}}}
\|b\|_{\dot{B}_{p,1}^{\sigma}}.
\end{equation*}
\end{lemma}
\begin{proof}
(1) When $s<1$.
From the definitions of  paraproduct and  commutator, we obtain
\[
[\dot{S}_{k_{0}}A(D),T_{a}]b = \sum_{j\in\mathbb{Z}} 
[\dot{S}_{k_{0}}A(D),\dot{S}_{j-1}a] \dot{\Delta}_{j}b.
\]
Since $A(D)$ is homogeneous of degree zero, we can use the localization properties of $\dot{S}_{k}$ and $\dot{\Delta}_{k}$ to choose a smooth function $\tilde{\phi}$ such that for all $j \leq k_{0} - 4$,
\[
[\dot{S}_{k_{0}}A(D),\dot{S}_{j-1}a]\dot{\Delta}_{j}b = \sum_{k\leq j-2} [\tilde{\phi}(2^{-j}D),\dot{\Delta}_{k}a] \dot{\Delta}_{j}b.
\]
By applying Lemma 2.97 from \cite{BaChDa11} under the condition that $q \leq p$, we obtain
\begin{equation}\label{eq:6.76}
\left\lVert[\tilde{\phi}(2^{-j}D),\dot{\Delta}_{k}a] \dot{\Delta}_{j}b\right\rVert_{L^{q}}
 \lesssim 2^{-j} \left\lVert\dot{\Delta}_{k}\nabla a\right\rVert_{L^{\frac{pq}{p-q}}} \left\lVert\dot{\Delta}_{j}b\right\rVert_{L^{p}}.
\end{equation}
For $|j - k_{0}| \leq 4$, there exists a smooth annularly-supported $\psi$ such that
\begin{equation*}
[\dot{S}_{k_{0}}A(D),\dot{S}_{j-1}a]\dot{\Delta}_{j}b 
= [\psi(2^{-k_{0}}D),\dot{S}_{j-1}a] \dot{\Delta}_{j}b,
\end{equation*}
which again satisfies \eqref{eq:6.76}. Restricting to $s<1$ and $p\leq2q$, and using Lemma 2.23 from \cite{BaChDa11} with convolution inequalities, we derive
\begin{equation*}
\begin{aligned}
&\|[\dot S_{k_{0}}A(D),T_{a}]b\|_{\dot{B}_{q,1}^{\sigma+s}}
\lesssim \sum_{j\in\mathbb{Z}}2^{j(\sigma+s)}\lVert 
[\dot{S}_{k_{0}}A(D),\dot{S}_{j-1}a] \dot{\Delta}_{j}b\rVert_{L^{q}}\\
&\lesssim \sum_{j\in\mathbb{Z}}\sum_{k\leq j-2}2^{(j-k)(s-1)} 2^{k(s-1)}\lVert\dot{\Delta}_{k}\nabla 
a \rVert_{L^{\frac{pq}{p-q}}} 2^{j\sigma}\lVert\dot{\Delta}_{j}b\rVert_{L^{p}}\\
&\lesssim \lVert\nabla a\rVert_{\dot{B}_{\frac{pq}{p-q},1}^{s-1}}
\lVert b\rVert_{\dot{B}_{p,\infty}^{\sigma}}\lesssim \lVert\nabla a\rVert_{\dot{B}_{p,1}^{s-1+\frac{2d}{p}-\frac{d}{q}}}
\lVert b\rVert_{\dot{B}_{p,1}^{\sigma}}.\\
\end{aligned}
\end{equation*}

(2) When $s=1$. Analogous to step 1, we have
\begin{equation*}
\begin{aligned}
&\left\lVert[\dot{S}_{k_0}A(D), T_a]b\right\rVert_{\dot{B}_{q,1}^{\sigma+s}} 
\lesssim \sum_{j\in\mathbb{Z}}\sum_{k\leq j-2}\left\lVert\dot{\Delta}_{k}\nabla a\right\rVert_{L^{\frac{pq}{p-q}}} 2^{j\sigma}\left\lVert\dot{\Delta}_{j}b\right\rVert_{L^{p}} \\
&\lesssim \left\lVert\nabla a\right\rVert_{\dot{B}_{\frac{pq}{p-q},1}^{0}}
\left\lVert b\right\rVert_{\dot{B}_{p,1}^{\sigma}} \lesssim \left\lVert\nabla a\right\rVert_{\dot{B}_{p,1}^{\frac{2d}{p}-\frac{d}{q}}}
\left\lVert b\right\rVert_{\dot{B}_{p,1}^{\sigma}}.
\end{aligned}
\end{equation*}
The desired results follow from these estimates.
\end{proof}

\section*{Acknowledgement}

This work was started while the first named author was visiting RIMS of Kyoto University. He would like to thank RIMS for the great hospitality. The authors are grateful to Professors Kenji Nakanishi and Nobu Kishimoto for the precious comments. Z. Guo is the recipient of an Australian Research Council Future Fellowship (project number FT230100588) funded by the Australian Government. M. Yang is supported by the National Natural Science Foundation of China (Grant No. 12161041) and the Jiangxi Province Natural Science Foundation (Grant No. 20252BAC250004).

\onehalfspacing


\begin{thebibliography}{99}

\bibitem{BaChDa11} Bahouri, H., Chemin, J.Y., Danchin, R.: {F}ourier Analysis and Nonlinear Partial Differential Equations. Springer, Heidelberg (2011).
\bibitem {FD2010} Charve, F., Danchin, R.: A global existence result for the compressible Navier-Stokes equations in the critical $L^{p}$ framework. Arch. Rational Mech. Anal., {\bf 198}, 233--271 (2010).
\bibitem {CMZ2010R} Chen, Q., Miao, C., Zhang, Z.: Well-posedness in critical spaces for the compressible Navier-Stokes equations with density-dependent viscosities. Rev. Mat. Iberoam, {\bf 26}, 1173--1224 (2010).
\bibitem {CMZ2010} Chen, Q., Miao, C., Zhang, Z.: Global well-posedness for compressible Navier-Stokes equations with highly oscillating initial velocity. Comm. Pure Appl. Math. {\bf 63}, 1173--1224 (2010).  
\bibitem {CMZ2015}Chen, Q., Miao, C., Zhang, Z.: On the ill-posedness of the compressible Navier-Stokes equations. Rev. Mat. Iberoam, {\bf 31}, 1375--1402 (2015).
\bibitem {D2001}  Danchin, R.: Local theory in critical spaces for compressible viscous and heat-conductive gases.  Commun. Partial Differ. Equ., {\bf 26}, 1183--1233 (2001).
\bibitem {D2005}Danchin, R.: On the uniqueness in critical spaces for compressible Navier–Stokes equations. NoDEA Nonlinear Differ. Equ. Appl. {\bf 12(1)}, 111--128 (2005).
\bibitem {D2000}Danchin, R.:  Global existence in critical spaces for compressible Navier-Stokes equations. Invent Math. {\bf 141}, 579--614 (2000).
\bibitem {D2014}  Danchin, R.: A Lagrangian approach for the compressible Navier-Stokes equations. Ann. Inst. Fourier {\bf 64(2)}, 753--791 (2014).
\bibitem {RH2015} Danchin, R.: {F}ourier analysis methods for the compressible Navier-Stokes equations. arXiv: 1507. 02637, (2015).
\bibitem {DH2016} Danchin, R, He, L.: The incompressible limit in $L^{p}$ type critical spaces. Math. Ann., {\bf 366(3)}, 1365--1402 (2016)
\bibitem {DX2017} Danchin, R., Xu, J.: Optimal Time-decay Estimates for the Compressible Navier–Stokes equations in the critical $L^{p}$ framework. Arch Rational Mech Anal. {\bf 224}, 53--90 (2017).
\bibitem {FK1964} Fujita, H., Kato, T.: On the Navier-Stokes initial value problem I. Arch. Rational Mech. Anal.  {\bf 16}, 269--315 (1964).
\bibitem {GYZ2024}  Guo, Z., Yang, M., Zhang, Z.: On the well-posedness of the compressible Navier-Stokes equations. arxiv:2409.01031, 2024.
\bibitem {H20111} Haspot, B.: Well-posedness in critical spaces for the system of compressible Navier–
Stokes in larger spaces. J. Differential Equations, {\bf 251}, 2262--2295 (2011).
\bibitem {H20112} Haspot, B.:  Existence of global strong solutions in critical spaces for barotropic viscous fluids. Arch. Ration Mech. Anal. {\bf 202}, 427--460 (2011).
\bibitem{HoZu1995} Hoff D. and Zumbrun K., Multi-dimensional diffusion waves for the Navier-Stokes equations of compressible flow. \textit{Indiana Univ. Math. J.}, {\bf{44}}, 604--676 (1995).
\bibitem{HoZu1997}
Hoff D. and Zumbrun K., Pointwise decay estimates for multidimensional Navier-Stokes diffusion waves. \textit{Z. Angew. Math. Phys.}, {\bf{48}},  597--614 (1997).
\bibitem {IO2022}  Iwabuchi, T., Ogawa, T.: Ill-posedness for the compressible Navier-Stokes equations under barotropic condition in limiting Besov spaces. Journal of the Mathematical Society of Japan {\bf 74(2)}, 353--394 (2022).
\bibitem {PL1998}  Lions, P. L.: Mathematical Topics in Fluid Mechanics: Vol 2: Compressible Models. Oxford: Oxford University Press, 1998.
\bibitem {MN1979} Matsumura A., Nishida T.: The initial value problem for the equations of motion of compressible viscous and heat-conductive fluids. Proc. Japan Acad., Ser. A {\bf 55}, 337--342 (1979).
\bibitem {MN1980} Matsumura A., Nishida T.: Initial boundary value problems for the equations of motion of compressible viscous and heat-conductive fluids. Communications in Mathematical Physics, {\bf 89}: 445--464 (1983).
\bibitem {Miyachi} Miyachi A.: On some singular Fourier multipliers for $H^p (\R^n)$, J. Fac. Sci. Univ. Tokyo Sect. IA Math. {\bf 28}, 267--315 (1981). 
\bibitem{Nash} Nash J.: Le problème de Cauchy pour les équations différentielles d'un fluide général, Bulletin
de la Soc. Math. de France, {\bf 90}, 487--497 (1962).
\bibitem {Peral} Peral J. C.: $L^p$ estimates for the wave equation. J. Funct. Anal., {\bf 6(1)}, 114--145 (1980).
\bibitem {X1998} Xin Z.: Blowup of smooth solutions to the compressible Navier‐Stokes equation with compact density. Communications on Pure and Applied Mathematics, , {\bf 51}, 229--240 (1998).    

\end{thebibliography}
\end{document}